\documentclass[journal,twoside,web]{ieeecolor}
\usepackage{generic}
\usepackage{cite}
\usepackage{amsmath,amssymb,amsfonts}
\usepackage{algorithmic}
\usepackage{graphicx}
\usepackage{circuitikz}
\usepackage{algorithm,algorithmic}
\usepackage{hyperref}
\hypersetup{hidelinks=true}
\usepackage{textcomp}

\usepackage{tikz}
\usetikzlibrary{arrows.meta, positioning}
\usepackage{pgfplots}
\pgfplotsset{compat=1.18}

\let\oldlabelindent\labelindent  % Save the original definition of \labelindent
\let\labelindent\relax  % relax the original definition of \labelindent
\usepackage{enumitem}  % Load enumitem
\let\labelindent\oldlabelindent

\newtheorem{definition}{Definition}

\newtheorem{theorem}{Theorem}
\newtheorem{lemma}{Lemma}
\newtheorem{corollary}{Corollary}
\newtheorem{proposition}{Proposition}
\newtheorem{assumption}{Assumption}
\newtheorem{problem}{Problem}
\newtheorem{example}{Example}

\DeclareMathOperator{\vecrz}{vec}

\DeclareMathOperator{\trace}{tr}
\def\BibTeX{{\rm B\kern-.05em{\sc i\kern-.025em b}\kern-.08em
    T\kern-.1667em\lower.7ex\hbox{E}\kern-.125emX}}
\begin{document}
\title{A New Noise Model for Data-driven Control:\\ Generalized Frobenius Norm Bounds}
\author{Huayuan Huang, M. Kanat Camlibel, \IEEEmembership{Senior Member, IEEE}, and Henk J. van Waarde, \IEEEmembership{Member, IEEE}
\thanks{This work was supported by the China Scholarship Council 202106340040. Henk J. van
Waarde acknowledges financial support by the Dutch Research Council (NWO) under the Talent Programme Veni Agreement (VI.Veni.22.335). \textit{(Corresponding author: Huayuan Huang.)}}
\thanks{The authors are with the Bernoulli Institute for Mathematics, Computer Science, and Artificial Intelligence, University of Groningen, 9747 AG Groningen, The Netherlands (email: huayuan.huang@rug.nl; m.k.camlibel@rug.nl; h.j.van.waarde@rug.nl)}}

\maketitle
\begin{abstract}
In this article, we introduce a new noise model for data-driven control. The model can be interpreted as a generalization of a Frobenius norm bound on the matrix of noise samples. For instantaneously bounded noise, the proposed model provides a less conservative overapproximation than an existing noise model based on a quadratic matrix inequality (QMI). Using the new model, we derive necessary and sufficient conditions for data-driven control. The framework covers a broad class of design problems, including quadratic stabilization, $\mathcal{H}_2$ control and $\mathcal{H}_{\infty}$ control, and is further extended to cover data-driven analysis problems, ranging for stabilizability to dissipativity. A key technical contribution is a new type of S-lemma that offers necessary and sufficient conditions under which a quadratic matrix inequality is implied by a quadratic inequality in vectorized variables.
\end{abstract}

\begin{IEEEkeywords}
Data-driven control, data informativity, LMIs, robust control. 
\end{IEEEkeywords}

\section{Introduction}
Research on data-driven control has gained significant attention in recent years \cite{hou2013modeltodata, markovsky2021behavioral, berberich2022combining, dorfler2022bridging, dorfler2023data, Henkbook}. This field can be broadly classified into methods that combine system identification with model-based control and methods that directly map data to control problems without an intermediate modeling step. This work concerns the latter category, which is known as direct data-driven control. By bypassing explicit model identification, direct data-driven control offers a robust alternative to model-based methods, particularly when the system cannot be uniquely identified from the available data. Its applications span stabilization \cite{DDCFormulas}, output regulation \cite{trentelman2021algebraicregulator}, model predictive control \cite{berberich2021data} and dissipativity analysis \cite{dissipativity2022}. In this context, a key challenge is that the available data are corrupted by unknown noise \cite{DDCFormulas, DDCStuggatt2020, dai2023data, martin2024Pointwiseboundednoise}. The present paper studies both analysis and design problems using noisy data.

In the literature on data-driven control, a common approach is to work with deterministic bounded noise. In this context, a substantial body of work assumes that the matrix of noise samples satisfies a QMI \cite{DDCStuggatt2020, verhoek2024LPV, eising2024continuous, BIANCHI2025switched}. Under this QMI-based characterization of noise, \cite{DDCmatrixSlemma} establishes data-based linear matrix inequality (LMI) conditions for a class of design problems that can be formulated as QMIs. Key advantages of this approach are the development of necessary and sufficient conditions for data-driven control, and the fact that the conditions can be checked efficiently because the size of the involved decision variables are independent of the time horizon of the experiment. Beyond the scope of \cite{DDCmatrixSlemma}, \cite{steentjes2022crosscovariance} considers cross-covariance bounded noise, while \cite{bisoffi2022data} and \cite{DDCQMI} study scenarios in which the generalized Slater condition required in \cite{DDCmatrixSlemma} is not satisfied. Subsequently, \cite{bisoffi2024controller} considers measurement errors and \cite{dataperturbation} introduces more general data perturbations. Moreover, \cite{kristovic2024dualization} and \cite{nakano2025dissipativity} study design problems using dissipativity techniques. In addition to design problems, also analysis problems have been studied, ranging from stabilizability \cite{stabilizability2022} and detectability \cite{detectability2022} to dissipativity \cite{koch2021provablydissipativity, dissipativity2022}.

Besides QMI-based characterizations, an alternative model assumes that the noise samples are bounded at every time step. Such instantaneous noise models are considered in several works \cite{martin2021dissipativity, bisoffi2024controller, martin2024Pointwiseboundednoise}, that impose a Euclidean norm bound on each noise sample. As of yet, there are no tractable necessary and sufficient conditions for data-driven analysis and control available with respect to this particular noise model. As noted in \cite{DDCmatrixSlemma} and further studied in \cite{bisoffi2021trade}, a (lossy) S-procedure \cite{SurveySlemma} can be applied to establish sufficient conditions for data-driven stabilization. However, a potential limitation of this approach is that the number of decision variables scales with the number of data samples \cite{martin2021dissipativity, bisoffi2021trade}. Another class of instantaneous noise bounds is considered in \cite{dai2023data, miller2023positive, DDCDensity}, which investigate noise samples whose infinity norm is bounded. The resulting conditions suffer from a similar limitation, as the dimensions of the decision variables depend on the time horizon of the experiment.

Beyond the aforementioned models, another fundamental characterization of noise is based on energy constraints that bound the \emph{sum} of squared Euclidean norms of the noise samples. Although such constraints have been widely studied in several fields, including state estimation \cite{bertsekas1971recursive} and set-membership identification \cite{fogel1979system, lahr2026kernelboundenergy}, they have received limited attention in the recent literature on data-driven control. These energy constraints over a finite time horizon are equivalent to a Frobenius norm bound on the noise matrix. Existing QMI-based models capture spectral norm bounds on the noise matrix, resulting in a conservative overapproximation of the true energy constraints. Instantaneous models generally fail to capture the collective properties of a noise sequence. This motivates the study of data-driven control under noise models based on Frobenius norm bounds. This problem is much more challenging than data-driven control with QMI-based noise models. In fact, stability and performance guarantees are naturally expressed in terms of QMIs in the uncertain system matrices \cite{DDCmatrixSlemma}. However, Frobenius norm bounds on the noise matrix lead to inequalities in terms of the \emph{vectorized} system matrices, rendering existing matrix S-lemmas \cite{DDCQMI} inapplicable.

In this paper, we introduce a new noise model that can be interpreted as a generalization of a Frobenius norm bound on the noise matrix. Within this framework, energy constraints are captured as a special case. We study data-driven analysis and control design for discrete-time linear time-invariant (LTI) systems with respect to this new noise model. In particular, we first consider a broad class of design problems that, for fixed control parameters, lead to QMIs in the system matrices. By specializing this setting, we derive necessary and sufficient conditions for quadratic stabilization, $\mathcal{H}_2$ control and $\mathcal{H}_{\infty}$ control. The resulting necessary and sufficient conditions are phrased in terms of data-based LMIs. The size of the involved decision variables is independent of the time horizon of the experiment, akin to approaches based on the QMI noise model. The framework is then extended to data-driven analysis. In particular, we provide necessary and sufficient conditions to verify quadratic stabilizability and dissipativity of all data-compatible systems. Finally, we compare our new noise model with existing ones. We show that the proposed model leads to a less conservative overapproximation of instantaneous noise bounds than QMI-based characterizations. The effectiveness of this new overapproximation will be demonstrated in numerical experiments.

A fundamental theoretical contribution of this work is the development of a new type of S-lemma, which serves as the mathematical cornerstone for our data-driven approaches. The new S-lemma provides necessary and sufficient conditions under which a QMI is implied by a quadratic inequality in vectorized variables. This theoretical result allows us to handle the generalized Frobenius norm structure of our noise model without sacrificing the necessity of the resulting analysis and design conditions.

To summarize, our main contributions are as follows:
\begin{enumerate}
    \item We introduce a new noise model for data-driven control, which is a generalization of Frobenius norm bounds. For instantaneously bounded noise, the new model provides a less conservative overapproximation than existing QMI-based models.
    \item We develop data-driven approaches under the new noise model for a broad class of control design problems, including quadratic stabilization, $\mathcal{H}_2$ control and $\mathcal{H}_{\infty}$ control.
    \item We extend the framework to data-driven analysis problems, providing necessary and sufficient conditions for quadratic stabilizability and dissipativity of all data-compatible systems.
    \item As an auxiliary result, we establish a new type of S-lemma that provides necessary and sufficient conditions under which a QMI is implied by a quadratic inequality in vectorized variables.
\end{enumerate}

The remainder of this paper is organized as follows. Section~\ref{section Preliminaries} discusses the preliminaries and the motivation for this work. The new noise model is presented in Section~\ref{section A new noise model}, followed by the reformulation of the set of data-compatible systems in Section~\ref{section The set of data compatible systems}. Section~\ref{section Methodology} develops the methodology, which is then applied to data-driven design and analysis in Sections~\ref{section Data-driven control design} and~\ref{section data driven analysis}, respectively. Section~\ref{section Illustrative examples} provides illustrative examples. Finally, Section~\ref{section Conclusion} concludes the paper.

\subsection*{Notation}
The Euclidean norm of a vector $y \in \mathbb{R}^p$ is denoted by $\|y\|_2$, while the spectral and Frobenius norms of a matrix $Y \in \mathbb{R}^{p \times q}$ are denoted by $\|Y\|_{2}$ and $\|Y\|_{\mathrm{F}}$, respectively. Moreover, $Y^{\dagger}$ denotes the \textit{Moore-Penrose pseudoinverse} of $Y$. Partition $Y = \begin{bmatrix}
    y_1 \!&\! y_2 \!&\! \cdots \!&\! y_q
\end{bmatrix}$, where $y_i \in \mathbb{R}^p$ for $i = 1, 2, \dots, q$. The vectorization of $Y$ is defined as $\vecrz(Y) = \begin{bmatrix}
	y_1^{\top} \!&\! y_2^{\top} \!&\! \cdots \!&\! y_q^{\top}
\end{bmatrix}^{\top}$. The Kronecker product of matrices $A \in \mathbb{R}^{k \times l}$ and $B \in \mathbb{R}^{m\times n}$ is denoted by $A \otimes B \in \mathbb{R}^{km\times ln}$. Let $X \in \mathbb{R}^{l\times m}$. Then, 
\begin{equation} \label{notation vec property}
    \vecrz(AXB) = (B^{\top} \otimes A) \vecrz (X).
\end{equation}
For a square matrix $Q \in \mathbb{R}^{q \times q}$, we denote its trace by $\trace(Q)$. We then have
\begin{equation} \label{notation tr property}
    \trace(A^{\top}B) = (\vecrz(A))^{\top}\vecrz(B),
\end{equation}
for any $A,B \in \mathbb{R}^{p \times q}$. The set of $n \times n$ real symmetric matrices is denoted by $\mathbb{S}^n$. We denote by $\mathbb{S}^{p,q}$ the set of all symmetric matrices
$$
\Pi = \begin{bmatrix}
    \Pi_{11} & \Pi_{12}\\
    \Pi_{21} & \Pi_{22}
\end{bmatrix},
$$
where $\Pi_{11} \in \mathbb{S}^{p}$ and $\Pi_{22} \in \mathbb{S}^{q}$. Finally, we define a set of matrices
$$
\begin{aligned}
    \mathbf{\Pi}_{p,q} \!:= \!\{
		\Pi \in \mathbb{S}^{p,q} \mid \Pi_{22} \!\leq 0, \ker{\Pi_{22}}\!\subseteq \ker{\Pi_{12}}, \Pi | \Pi_{22}\! \geq 0 \}\!,
\end{aligned}
$$
where $\Pi | \Pi_{22} := \Pi_{11} - \Pi_{12}\Pi_{22}^{\dagger}\Pi_{21}$ is the generalized Schur complement of $\Pi$ with respect to $\Pi_{22}$. 

\section{Preliminaries} \label{section Preliminaries}
In this section, we present the preliminaries. Section~\ref{section System data and noise} introduces the system, data and noise. Section~\ref{section Informativity for M control} defines informativity for $M$-control and reviews related results. Section~\ref{section Motivation} discusses the motivation of this paper.

\subsection{System, data and noise} \label{section System data and noise}
Consider the discrete-time LTI system
\begin{equation} \label{eq real system}
    x(t+1) = A_sx(t)+B_su(t)+w(t),
\end{equation}
where $x(t) \in \mathbb{R}^n$ is the state, $u(t) \in \mathbb{R}^m$ is the input and $w(t) \in \mathbb{R}^n$ is the process noise. The parameter matrices $A_s \in \mathbb{R}^{n \times n}$ and $B_s \in \mathbb{R}^{n \times m}$ are unknown. Instead, we are given the input data $u(0),u(1),\dots, u(T-1)$ and the state data $x(0),x(1),\dots, x(T)$, which are collected in the matrices
$$
\begin{aligned}
	U_- &:= \begin{bmatrix}
		u(0) & u(1) & \cdots & u(T-1)
	\end{bmatrix},\\
    X &:= \begin{bmatrix} 
            x(0) & x(1) & \cdots & x(T)
	\end{bmatrix}.
\end{aligned}
$$
These data are generated by the true system $(A_s,B_s)$ influenced by unknown noise $w(0),w(1),\dots, w(T-1)$, such that
\begin{equation}\label{realdataeq DT}
    X_+ = A_{s}X_- + B_sU_- + W_-,
\end{equation}
where
$$
    \begin{aligned}
	X_+ &:= \begin{bmatrix} 
            x(1) & x(2) & \cdots & x(T)
	\end{bmatrix},\\
	X_- &:= \begin{bmatrix} 
            x(0) & x(1) & \cdots & x(T-1)
	\end{bmatrix},\\
	W_- &:= \begin{bmatrix}
			w(0) & w(1) & \cdots & w(T-1)
	\end{bmatrix}.
    \end{aligned}
$$
Although the noise matrix $W_-$ is unknown, we assume that it belongs to a given set $\mathcal{W}$, i.e., $W_- \in \mathcal{W}$. 

\begin{example}
A natural choice of $\mathcal{W}$ captures instantaneous bounds on the noise. In this case, $\mathcal{W}$ is given by
\begin{equation} \label{definition Wib}
    \mathcal{W}_{\mathrm{ib}}(\epsilon) \!:=\! \{W \!\in \mathbb{R}^{n \times T} \!\mid\! \|W_i\|_2^2 \leq \epsilon \text{ for }i = 1,2,\dots,T\},
\end{equation}
where $W_i$ denotes the $i$-th column of $W$ and $\epsilon \geq 0$ is given. This can be expressed as an Euclidean norm bound on each noise sample, i.e., 
$$
    \|w(t)\|_2^2 \leq \epsilon,\text{ for }t = 0,1,\dots,T-1.
$$
An alternative noise model is given by
\begin{equation} \label{definition Wqmi}
    \mathcal{W}_{\mathrm{qmi}}(\Psi):= \bigg\{ W \in \mathbb{R}^{n \times T} \mid \begin{bmatrix}
    I \\ W^{\top}
\end{bmatrix}^{\top} \Psi \begin{bmatrix}
    I \\ W^{\top}
\end{bmatrix} \geq 0 \bigg\},
\end{equation}
where $\Psi \in \mathbf{\Pi}_{n,T}$. This model is governed by a QMI and is widely considered in recent works.
\end{example}

A system $(A,B)$ is called \emph{compatible} with the data $\mathcal{D}:=(U_-,X)$ if
\begin{equation}\label{dataeq DT}
    X_+ = AX_- + BU_- + W_-,
\end{equation}
for some $W_- \in \mathcal{W}$. We denote the set of all systems compatible with $\mathcal{D}$ by $\Sigma_{\mathcal{D}}(\mathcal{W})$, i.e.,
\begin{equation} \label{definition Sigma}
    \Sigma_{\mathcal{D}}(\mathcal{W}): = \{(A,B)\mid \eqref{dataeq DT}\text{ holds for some }W_- \in \mathcal{W}\}.
\end{equation}
It follows from \eqref{realdataeq DT} that $(A_s,B_s) \in \Sigma_{\mathcal{D}}(\mathcal{W})$, but in general $\Sigma_{\mathcal{D}}(\mathcal{W})$ contains other systems. 

\subsection{Informativity for $\mathit{M}$-control} \label{section Informativity for M control}
In data-driven control, a central problem is the design of a controller for the true system. Since on the basis of the given data we cannot distinguish $(A_s,B_s)$ from any other system in $\Sigma_{\mathcal{D}}(\mathcal{W})$, the goal is to find a single controller that works for all systems in $\Sigma_{\mathcal{D}}(\mathcal{W})$. In the following, we introduce the notion of informativity for a class of data-driven control problems that can be formulated as QMIs in $\begin{bmatrix}
    A & B
\end{bmatrix}$.

\begin{definition} \label{definition 1}
    Let $\Theta$ be a set of control parameters, and consider a mapping $M: \Theta \rightarrow  \mathbb{S}^{n,n+m}$ with $M_{22}(\theta) \leq 0$ for all $\theta \in \Theta$. The data $\mathcal{D}$ are called \textit{informative for $M$-control with respect to the noise model $\mathcal{W}$} if there exists a $\theta \in \Theta$ such that
    \begin{equation} \label{definition 1 M theta A B}
        \begin{bmatrix} I \\ A^\top \\ B^\top \end{bmatrix}^\top M(\theta) 
        \begin{bmatrix} I \\ A^\top \\ B^\top \end{bmatrix} > 0,
    \end{equation}
    for all $(A,B) \in \Sigma_{\mathcal{D}}(\mathcal{W})$.
\end{definition}

With Definition~\ref{definition 1}, we can investigate several data-driven control problems by specifying the pair $(\Theta,M)$ accordingly. 

\begin{example}
Data-driven quadratic stabilization arises as a special case by taking $(\Theta,M) =(\Theta_{\mathrm{stab}}, M_{\mathrm{stab}})$, where
\begin{equation} \label{definition Theta stab}
    \Theta_{\mathrm{stab}} :=\{(P,K) \in \mathbb{S}^n \times \mathbb{R}^{m \times n} \mid P > 0\},
\end{equation}
and
\begin{equation} \label{definition M stab}
    M_{\mathrm{stab}}(P,K) := \begin{bmatrix}
    P & 0 & 0\\
    0 & -P & -PK^{\top}\\
    0 & -KP & -KPK^{\top}
\end{bmatrix}.
\end{equation}
In addition, data-driven $\mathcal{H}_{2}$ control and $\mathcal{H}_{\infty}$ control can also be captured by appropriate choices of $(\Theta,M)$.
\end{example}

While the general problem of $M$-control has not been studied explicitly, existing results on data-driven quadratic stabilization can be readily extended to $M$-control. For $\mathcal{W}_{\mathrm{ib}}(\epsilon)$, although no necessary and sufficient conditions are available, sufficient conditions can be derived by employing the (lossy) S-procedure \cite{boyd1994linear}, as observed in \cite[Sec.~VII]{DDCmatrixSlemma} and further investigated in \cite{bisoffi2021trade}. For $\mathcal{W}_{\mathrm{qmi}}(\Psi)$, \cite{DDCQMI} provides necessary and sufficient conditions for quadratic stabilization. For the sake of clarity, we recap both approaches, extended to the general $M$-control problem, in the following.

\begin{proposition} \label{proposition preliminaries}
Let $\epsilon \geq 0$, $\Psi \in \mathbf{\Pi}_{n,T}$, $\Theta$ be a set of control parameters and $M: \Theta \rightarrow \mathbb{S}^{n,n+m}$ be a mapping with $M_{22}(\theta) \leq 0$ for all $\theta \in \Theta$. Then the following statements hold.
\begin{enumerate} [label=(\roman*)]
\item \label{proposition statement multiple multipliers}  The data $\mathcal{D}$ are informative for $M$-control with respect to $\mathcal{W}_{\mathrm{ib}}(\epsilon)$ if there exist a $\theta \in \Theta$ and scalars $\alpha_0,\alpha_1, \dots, \alpha_{T-1} \geq 0$ and $\beta > 0$ such that 
\begin{equation} \label{proposition multiple multipliers LMI}
    \!M(\theta)-\! \sum_{t=0}^{T-1} \!\alpha_t\!\! \begin{bmatrix}
    I \!\!&\!\!\! x(t\!+\!1)\!\\
    0 \!\!&\!\!\! -x(t)\!\\
    0 \!\!&\!\!\! -u(t)\!
\end{bmatrix}\!\!\! \begin{bmatrix}
    \epsilon I \!\!&\!\!\! 0 \\
    0 \!\!&\!\!\!\! -I
\end{bmatrix}\!\!\! \begin{bmatrix}
    I \!\!&\!\!\! x(t\!+\!1)\!\\
    0 \!\!&\!\!\! -x(t)\!\\
    0 \!\!&\!\!\! -u(t)\!
\end{bmatrix}^{\!\!\!\top}\!\! \geq\!\! \begin{bmatrix}
    \beta I_{n} \!\!&\!\!\! 0\\
    0 \!\!&\!\!\! 0
\end{bmatrix}\!\!.
\end{equation}
\item \label{proposition statement a QMI set}  The data $\mathcal{D}$ are informative for $M$-control with respect to $\mathcal{W}_{\mathrm{qmi}}(\Psi)$ if and only if there exist a $\theta \in \Theta$ and scalars $\alpha \geq 0$ and $\beta > 0$ such that
\begin{equation} \label{proposition statement a QMI set LMI}
\!M(\theta) - \alpha \begin{bmatrix}
    I \!\!&\!\! X_+ \\
    0 \!\!&\!\! -X_- \\
    0 \!\!&\!\! -U_-
\end{bmatrix} \Psi \begin{bmatrix}
    I \!\!&\!\! X_+ \\
    0 \!\!&\!\! -X_- \\
    0 \!\!&\!\! -U_-
\end{bmatrix}^{\!\!\top} \!\geq\! \begin{bmatrix}
    \beta I_n \!\!&\!\! 0\\
    0 \!\!&\!\! 0
\end{bmatrix}.
\end{equation}
\end{enumerate}
\end{proposition}
\vspace{1em}

The proof of Proposition~\ref{proposition preliminaries} is presented in the appendix.

\subsection{Motivation} \label{section Motivation}
In Proposition \ref{proposition preliminaries}, statement \ref{proposition statement multiple multipliers} provides sufficient conditions for informativity for $M$-control with respect to $\mathcal{W}_{\mathrm{ib}}(\epsilon)$, but the resulting LMI \eqref{proposition multiple multipliers LMI} relies on $T$ multipliers, namely $\alpha_0,\alpha_1,\dots,\alpha_{T-1}$. This is a potential drawback since the number of decision variables depends on the number of data samples. Alternatively, informativity with respect to $\mathcal{W}_{\mathrm{ib}}(\epsilon)$ can be guaranteed by verifying the LMI \eqref{proposition statement a QMI set LMI} in statement \ref{proposition statement a QMI set} with $\Psi = \Psi_{\epsilon}$, where
\begin{equation} \label{definition Psi epsilon}
\Psi_{\epsilon} := \begin{bmatrix}
    \epsilon T I & 0 \\
    0 & -I
\end{bmatrix}.
\end{equation}
Indeed, any $W \in \mathcal{W}_{\mathrm{ib}}(\epsilon)$ satisfies 
$$
WW^{\top} = \sum_{i=1}^TW_iW_i^{\top}  \leq \epsilon T I,
$$
which implies $\mathcal{W}_{\mathrm{ib}}(\epsilon) \subseteq \mathcal{W}_{\mathrm{qmi}}(\Psi_{\epsilon})$. Therefore, the necessary and sufficient conditions for the informativity with respect to $\mathcal{W}_{\mathrm{qmi}}(\Psi_{\epsilon})$ are sufficient for the informativity with respect to $\mathcal{W}_{\mathrm{ib}}(\epsilon)$. Since \eqref{proposition statement a QMI set LMI} involves only one multiplier $\alpha$, the number of decision variables does not depend on the number of data samples, meaning that \eqref{proposition statement a QMI set LMI} is favorable from a computational point of view. However, it can be verified that if $\theta$, $\alpha$ and $\beta$ satisfy \eqref{proposition statement a QMI set LMI} with $\Psi = \Psi_{\epsilon}$, then \eqref{proposition multiple multipliers LMI} holds with the same $\theta$ and $\beta$ and with $\alpha_0=\alpha_1=\dots=\alpha_{T-1}=\alpha$. Hence, the conditions based on \eqref{proposition statement a QMI set LMI} can be more conservative than those based on \eqref{proposition multiple multipliers LMI}. 

To conclude, compared to the approach relying on $T$ multipliers, the $\mathcal{W}_{\mathrm{qmi}}(\Psi_{\epsilon})$-based formulation offers computational advantages, albeit at the cost of being a conservative overapproximation of $\mathcal{W}_{\mathrm{ib}}(\epsilon)$. In this paper, we will introduce a new noise model and provide necessary and sufficient conditions for data-driven control with respect to this noise model. In the context of $\mathcal{W}_{\mathrm{ib}}(\epsilon)$, the results of this paper are a viable alternative to the approaches based on $T$ multipliers and $\mathcal{W}_{\mathrm{qmi}}(\Psi_{\epsilon})$. Indeed, as we will show (Lemma~\ref{lemma W set inclusion}), the new noise model provides a less conservative overapproximation of $\mathcal{W}_{\mathrm{ib}}(\epsilon)$ than $\mathcal{W}_{\mathrm{qmi}}(\Psi_{\epsilon})$, while the novel resulting method preserves the same computational benefits as the one based on $\mathcal{W}_{\mathrm{qmi}}(\Psi_{\epsilon})$.

\section{A new noise model} \label{section A new noise model}
We define the noise model
\begin{equation} \label{noise bound vector}
    \mathcal{W}_{\mathrm{F}}(\Phi)\!:=\! \bigg\{ W \!\in \mathbb{R}^{n \times T} \!\mid \!\begin{bmatrix}
    1 \\ \vecrz (W)
\end{bmatrix}^{\!\!\top}\! \Phi \begin{bmatrix}
    1 \\ \vecrz (W)
\end{bmatrix} \!\!\geq\! 0 \bigg\}.
\end{equation}
where $\Phi \in \mathbf{\Pi}_{1,nT}$ is given. By \cite[Thm.~3.2]{DDCQMI}, the set $\mathcal{W}_{\mathrm{F}}(\Phi)$ is nonempty and convex. 

This model is governed by a quadratic inequality in the vectorized noise matrix and can be interpreted as a generalization of a Frobenius norm bound on $W$. Indeed, as a special case with $\Phi_{11} \geq 0$, $\Phi_{12} = \Phi_{21}^{\top} = 0$ and $\Phi_{22}=-I$, we have that $W \in \mathcal{W}_{\mathrm{F}}(\Phi)$ if and only if
$$
\|W\|_{\mathrm{F}}^2 = \|\vecrz (W)\|_2^2 \leq \Phi_{11}.
$$ 

The following lemma shows that, with a suitable choice of $\Phi$, the new noise model provides a less conservative overapproximation of $\mathcal{W}_{\mathrm{ib}}(\epsilon)$ than $\mathcal{W}_{\mathrm{qmi}}(\Psi_{\epsilon})$.

\begin{lemma} \label{lemma W set inclusion}
    Let $\epsilon \geq 0$, and let $\Psi_{\epsilon} \in \mathbb{S}^{n, T}$ be defined as in \eqref{definition Psi epsilon}. Define $\Phi_{\epsilon} \in \mathbb{S}^{1, nT}$ as
    \begin{equation} \label{definition Phi ib}
    \Phi_{\epsilon} := \begin{bmatrix}
        \epsilon T & 0 \\
        0 & -I
    \end{bmatrix}.
    \end{equation}
    Then,
    \begin{equation} \label{W inclusion}
        \mathcal{W}_{\mathrm{ib}}(\epsilon) \subseteq \mathcal{W}_{\mathrm{F}}(\Phi_{\epsilon}) \subseteq \mathcal{W}_{\mathrm{qmi}}(\Psi_{\epsilon}).
    \end{equation}
    Moreover, 
    $\mathcal{W}_{\mathrm{F}}(\Phi_{\epsilon}) = \mathcal{W}_{\mathrm{qmi}}(\Psi_{\epsilon})$ if and only if at least one of the following holds: $\epsilon = 0$, $n = 1$ or $T = 1$.
\end{lemma}
\begin{proof}
    We first prove the inclusion \eqref{W inclusion}. It holds that
    $$
    \mathcal{W}_{\mathrm{F}}(\Phi_{\epsilon}) = \{ W \in \mathbb{R}^{n \times T} \mid \|W\|_{\mathrm{F}}^2 \leq \epsilon T \}.
    $$
    Since $\|W\|_{\mathrm{F}}^2 = \sum_{i=1}^{T}\|W_i\|_2^2 \leq \epsilon T$ holds for all $W \in \mathcal{W}_{\mathrm{ib}}(\epsilon)$, we have $\mathcal{W}_{\mathrm{ib}}(\epsilon) \subseteq \mathcal{W}_{\mathrm{F}}(\Phi_{\epsilon})$. Note that $\mathcal{W}_{\mathrm{qmi}}(\Psi_{\epsilon})$ can be expressed in terms of a spectral norm bound as follows:
    $$
    \mathcal{W}_{\mathrm{qmi}}(\Psi_{\epsilon}) = \{ W \in \mathbb{R}^{n \times T} \mid \|W\|_{2}^2 \leq \epsilon T \}.
    $$
    Since $\|W\|_{2} \leq  \|W\|_{\mathrm{F}}$ for any matrix $W$, it follows that $\mathcal{W}_{\mathrm{F}}(\Phi_{\epsilon}) \subseteq \mathcal{W}_{\mathrm{qmi}}(\Psi_{\epsilon})$, which proves \eqref{W inclusion}. 

    We now prove the equivalence statement. The ``if" part follows directly from the fact that $\|W\|_2 = \|W\|_{\mathrm{F}}$ for any matrix $W$ with rank at most one. For the ``only if" part, assume that $\mathcal{W}_{\mathrm{F}}(\Phi_{\epsilon}) = \mathcal{W}_{\mathrm{qmi}}(\Psi_{\epsilon})$. Suppose, by contradiction, that $\epsilon > 0$, $n \geq 2$ and $T \geq 2$. There exists a $W \in \mathcal{W}_{\mathrm{qmi}}(\Psi_{\epsilon})$ with rank at least two. This implies $\|W\|_2 \neq 0$. Define $\bar{W} := (\epsilon T)^{\frac{1}{2}}\|W\|_2^{-1} W $. Then $\|\bar{W}\|_2^2 = \epsilon T$ and thus $\bar{W} \in  \mathcal{W}_{\mathrm{qmi}}(\Psi_{\epsilon})$. Note that the rank of $\bar{W}$ equals the rank of $W$. By \cite[Prop.~9.2.3]{bernstein2009matrix}, we have $\|\bar{W}\|_{2} <  \|\bar{W}\|_{\mathrm{F}}$ and therefore $\|\bar{W}\|_{\mathrm{F}}^2 > \epsilon T$. We conclude that $\bar{W} \notin \mathcal{W}_{\mathrm{F}}(\Phi_{\epsilon})$, which contradicts the assumption that $\mathcal{W}_{\mathrm{F}}(\Phi_{\epsilon}) = \mathcal{W}_{\mathrm{qmi}}(\Psi_{\epsilon})$. Hence, at least one of the following holds: $\epsilon = 0$, $n = 1$ or $T = 1$. This proves the lemma.    
\end{proof}

Lemma \ref{lemma W set inclusion} illustrates how $\mathcal{W}_{\mathrm{F}}(\Phi)$ relates to the well-studied noise models $\mathcal{W}_{\mathrm{ib}}(\epsilon)$ and $\mathcal{W}_{\mathrm{qmi}}(\Psi)$. This shows that $\mathcal{W}_{\mathrm{F}}(\Phi_{\epsilon})$ is a less conservative overapproximation of $\mathcal{W}_{\mathrm{ib}}(\epsilon)$ than $\mathcal{W}_{\mathrm{qmi}}(\Psi_{\epsilon})$, and in general $\mathcal{W}_{\mathrm{F}}(\Phi_{\epsilon}) \neq \mathcal{W}_{\mathrm{qmi}}(\Psi_{\epsilon})$. 

The primary goal of the paper is to study informativity for $M$-control with respect to the noise model $\mathcal{W}_{\mathrm{F}}(\Phi)$. In particular, we study the following problem.

\begin{problem} \label{problem 1}
    Find necessary and sufficient conditions under which the data $\mathcal{D}$ are informative for $M$-control with respect to $\mathcal{W}_{\mathrm{F}}(\Phi)$.
\end{problem}

Beyond control design, we will also extend the framework to data-driven analysis problems, including stabilizability and dissipativity. 

Finally, by virtue of Lemma~\ref{lemma W set inclusion}, our new conditions for informativity with respect to the noise model $\mathcal{W}_{\mathrm{F}}(\Phi)$ can also be used to obtain novel conditions for informativity with respect to the noise model $\mathcal{W}_{\mathrm{ib}}(\epsilon)$. In Section~\ref{section Discussion on the instantaneous noise bound} and the illustrative examples in Section~\ref{section Illustrative examples}, we will compare our approach to existing methods for this noise model.

In the remainder of the paper, when referring to the noise model $\mathcal{W}_{\mathrm{F}}(\Phi)$ without specifying a particular $\Phi \in \mathbf{\Pi}_{1,nT}$, we simply write $\mathcal{W}_{\mathrm{F}}$.

\section{The set of data-compatible systems} \label{section The set of data compatible systems}
In this section, we will focus on the system set $\Sigma_{\mathcal{D}}(\mathcal{W}_{\mathrm{F}})$. From the definition \eqref{definition Sigma},
$$
\Sigma_{\mathcal{D}}(\mathcal{W}_{\mathrm{F}}) = \{(A,B)\mid \eqref{dataeq DT}\text{ holds for some }W_- \in \mathcal{W}_{\mathrm{F}}\}.
$$
By vectorizing \eqref{dataeq DT}, 
$$
\vecrz \left({X}_+\right) = \left( \begin{bmatrix}
    X_- \\ U_-
\end{bmatrix}^{\top} \otimes I \right)\vecrz \left(\begin{bmatrix}
    A & B
\end{bmatrix}\right) + \vecrz \left(W_-\right).
$$
Define
\begin{equation} \label{definition N}
\!\! N\! := \!
    \begin{bmatrix}
     1 & 0 \\
    \!\vecrz(X_+)\! &\!\! -\begin{bmatrix}
    X_-\! \\ U_-\!
\end{bmatrix}^{\!\!\top}\!\! \otimes I
\end{bmatrix}^{\!\!\top}\!\Phi\!
\begin{bmatrix}
     1 & 0 \\
    \!\vecrz(X_+)\! &\!\! -\begin{bmatrix}
    X_-\! \\ U_-\!
\end{bmatrix}^{\!\!\top}\!\! \otimes I
\end{bmatrix}\!\!.
\end{equation}
Then, $(A,B) \in \Sigma_{\mathcal{D}}(\mathcal{W}_{\mathrm{F}})$ if and only if $\vecrz(\begin{bmatrix}
        A & B
    \end{bmatrix})$ satisfies
\begin{equation} \label{vec A B}
    \begin{bmatrix}
    1 \\ \vecrz(\begin{bmatrix}
        A & B
    \end{bmatrix})
\end{bmatrix}^{\top} N \begin{bmatrix}
    1 \\ \vecrz(\begin{bmatrix}
        A & B
    \end{bmatrix})
\end{bmatrix} \geq 0.
\end{equation}

In the following lemma, we show that the matrix $N$ belongs to the set $\mathbf{\Pi}_{1,n(n+m)}$.
\begin{lemma} \label{lemma N in Pi}
    Let $N \in \mathbb{S}^{1, n(n+m)}$ be defined as in \eqref{definition N}. Then $N \in \mathbf{\Pi}_{1,n(n+m)}$.
\end{lemma}
\begin{proof}
    Since $\Phi \in \mathbf{\Pi}_{1,nT}$, we have that $\Phi_{22} \leq 0$ and hence
$$
N_{22} = \left(\begin{bmatrix}
    X_- \\ U_-
\end{bmatrix} \otimes I \right) \Phi_{22}\left(\begin{bmatrix}
    X_- \\ U_-
\end{bmatrix}^{\top} \otimes I\right) \leq 0.
$$
Moreover, 
$$
\ker N_{22} = \ker \Phi_{22}\left(\begin{bmatrix}
    X_- \\ U_-
\end{bmatrix}^{\top} \otimes I\right).
$$
Since 
$$
N_{12} = \begin{bmatrix}
     1 \\
    \vecrz(X_+)
\end{bmatrix}^{\top}\begin{bmatrix}
    \Phi_{12} \\ \Phi_{22}
\end{bmatrix}
\left(-\begin{bmatrix}
    X_- \\ U_-
\end{bmatrix}^{\top}\!\! \otimes I \right),
$$
it follows from $\ker \Phi_{22} \subseteq \ker \Phi_{12}$ that $\ker N_{22} \subseteq \ker N_{12}$. Now it follows from $(A_s,B_s) \in \Sigma_{\mathcal{D}}(\mathcal{W}_{\mathrm{F}})$ that
$$
\begin{aligned}
N|N_{22} + (\vecrz(\begin{bmatrix}
        A_s & B_s
    \end{bmatrix}) + N_{22}^{\dagger}N_{21})^{\top}N_{22}&\\
    (\vecrz(\begin{bmatrix}
        A_s & B_s
    \end{bmatrix}) +& N_{22}^{\dagger}N_{21}) \geq 0,
\end{aligned}
$$
and thus $N|N_{22} \geq 0$. We conclude that $N \in \mathbf{\Pi}_{1,n(n+m)}$. 
\end{proof}

With Lemma \ref{lemma N in Pi}, we have shown that the system set $\Sigma_{\mathcal{D}}(\mathcal{W}_{\mathrm{F}})$ can be characterized by a quadratic inequality in $\vecrz(\begin{bmatrix}
    A & B
\end{bmatrix})$ with a structured matrix $N \in \mathbf{\Pi}_{1,n(n+m)}$. Therefore, Problem~\ref{problem 1} amounts to determining conditions under which there exists a $\theta \in \Theta$ such that \eqref{definition 1 M theta A B} holds for all $\vecrz(\begin{bmatrix}
    A & B
\end{bmatrix})$ satisfying \eqref{vec A B}. In other words, checking informativity requires us to check whether a QMI is implied by a quadratic inequality in vectorized variables. This motivates the following methodology section.

\section{Methodology} \label{section Methodology}
In what follows, we present a new type of S-lemma, providing conditions under which a QMI is implied by a quadratic inequality in vectorized variables. We first introduce some notation. For $\Pi \in \mathbb{S}^{p,q}$, we define
$$
\mathcal{Z}_{q}(\Pi):= \bigg\{ Z \in \mathbb{R}^{q \times p} \mid \begin{bmatrix}
    I \\ Z
\end{bmatrix}^{\top} \Pi \begin{bmatrix}
    I \\ Z
\end{bmatrix} \geq 0 \bigg\}.
$$
The counterpart corresponding to a strict QMI is denoted by $\mathcal{Z}_{q}^+(\Pi)$. 
By combining Theorem~4.10 and Corollary~4.13 in \cite{DDCQMI}, we obtain the following lemma.
\begin{lemma} \label{lemma QMI}
    Let $M \in \mathbb{S}^{p, q}$ and $N \in \mathbf{\Pi}_{p,q}$. Assume that either $M_{22} \leq 0$ or $N_{22} < 0$. Then, 
    $$
    Z \in \mathcal{Z}_{q}(N) \Rightarrow Z \in \mathcal{Z}_{q}^+(M),
    $$
    if and only if there exist scalars $\alpha \geq 0$ and $\beta > 0$ such that
    $$
    M - \alpha N \geq \begin{bmatrix}
        \beta I & 0\\
        0 & 0
    \end{bmatrix}.
    $$
\end{lemma}
\vspace{1em}
Building on this, we have the following proposition.
\begin{proposition} \label{proposition benchmark}
    Let $M \in \mathbb{S}^{p, q}$ and $N \in \mathbf{\Pi}_{1,qp}$. Assume that either $M_{22} \leq 0$ or $N_{22} < 0$. Then, 
    \begin{equation} \label{implication N to M}
        Z \in \mathbb{R}^{q \times p} \text{ and }\vecrz (Z^{\top}) \in \mathcal{Z}_{qp}(N) \Rightarrow Z \in \mathcal{Z}_{q}^+(M),
    \end{equation}
    if and only if there exist functions $\alpha,\beta: \mathbb{R}^p \rightarrow \mathbb{R}$ with $\alpha(x) \geq 0$ and $\beta(x)>0$ for all $x \in \mathbb{R}^p \setminus \{0\}$ such that
    \begin{equation} \label{proposition benchmark QMI}
    \begin{bmatrix}
        x^{\top}M_{11}x & x^{\top} M_{12} \otimes x^{\top}\\
        M_{21}x \otimes x & M_{22}\otimes xx^{\top}
    \end{bmatrix} - \alpha(x)
    N  \geq \begin{bmatrix}
        \beta(x) & 0 \\
        0 & 0
    \end{bmatrix},
    \end{equation}
    for all $x \in \mathbb{R}^p \setminus \{0\}$. 
\end{proposition}
\begin{proof}
    We first prove the ``only if" part. Suppose that \eqref{implication N to M} holds. This implies that  
    \begin{equation} \label{M Z x}
        x^{\top} \! M_{11}x + x^{\top}\!M_{12}Zx + x^{\top}\!Z^{\top}\!M_{21}x + x^{\top}\!Z^{\top}\!M_{22}Zx \!>\! 0,
    \end{equation}
    for all $x \in \mathbb{R}^p \setminus \{0\}$ and all $Z \in \mathbb{R}^{q \times p}$ with $\vecrz (Z^{\top}) \in \mathcal{Z}_{qp}(N)$. It follows from \eqref{notation vec property} that $Zx = (I_q \otimes x^{\top})\vecrz  (Z^{\top})$. We then have that the inequality in \eqref{M Z x} is equivalent to 
    \begin{equation} \label{proposition benchmark vectorization}
        \begin{bmatrix}
        1 \\ \vecrz  (Z^{\top})
    \end{bmatrix}^{\!\!\top} \!\!\begin{bmatrix}
        x^{\top}\!M_{11}x & x^{\top}\! M_{12} \otimes x^{\top}\\
        M_{21}x \otimes x\! & M_{22}\otimes xx^{\top}
    \end{bmatrix}\!\! \begin{bmatrix}
        1 \\ \vecrz  (Z^{\top})
    \end{bmatrix} \!\!>\! 0.
    \end{equation}
    Fix an arbitrary $x \neq 0$. By Lemma \ref{lemma QMI}, there exist scalars $\alpha(x) \geq 0$ and $\beta(x)>0$ such that \eqref{proposition benchmark QMI} holds. Since this argument holds for all $x \neq 0$, we conclude that there exist functions $\alpha,\beta: \mathbb{R}^p \rightarrow \mathbb{R}$ with $\alpha(x) \geq 0$ and $\beta(x)>0$ for all $x \neq 0$ such that \eqref{proposition benchmark QMI} holds for all $x \neq 0$. This proves the ``only if" part.
    
    We next prove the ``if" part. Suppose that there exist functions $\alpha,\beta: \mathbb{R}^p \rightarrow \mathbb{R}$ with $\alpha(x) \geq 0$ and $\beta(x)>0$ for all $x \neq 0$ such that \eqref{proposition benchmark QMI} holds for all $x \neq 0$. Let $Z \in \mathbb{R}^{q \times p}$ with $\vecrz (Z^{\top}) \in \mathcal{Z}_{qp}(N)$. Multiply \eqref{proposition benchmark QMI} from left by $\begin{bmatrix}
        1 & (\vecrz (Z^{\top}))^{\top}
    \end{bmatrix}$ and from right by its transpose. This yields \eqref{proposition benchmark vectorization}, and thus, \eqref{M Z x} holds for all $x \neq 0$. Since this argument holds for all $Z \in \mathbb{R}^{q \times p}$ with $\vecrz (Z^{\top}) \in \mathcal{Z}_{qp}(N)$, we conclude that the implication \eqref{implication N to M} holds. This proves the proposition.
\end{proof}

Proposition \ref{proposition benchmark} provides necessary and sufficient conditions under which \eqref{implication N to M} holds. However, this result is computationally difficult to verify. In particular, it is non-trivial to directly verify the matrix inequality \eqref{proposition benchmark QMI} for all $x \neq 0$. This motivates us to derive alternative conditions. As we will show later on, these conditions remain necessary and sufficient under an additional structural assumption while being computationally tractable. 

To state our main results, we need two technical lemmas. The following lemma introduces a new QMI-based set containing all matrices whose vectorization satisfy a quadratic inequality.

\begin{lemma} \label{lemma general inclusion}
    Let $N \in \mathbf{\Pi}_{1,qp}$ and assume that $N_{22} = Q_{22} \otimes I$ for some $Q_{22} \in \mathbb{S}^{q}$. Define $Q \in \mathbb{S}^{p, q}$ as 
    \begin{equation} \label{definition Theta}
        Q := \begin{bmatrix}
        (N | N_{22})I + Q_{12} Q_{22}^{\dagger} Q_{21} & Q_{12} \\
        Q_{21} & Q_{22}
    \end{bmatrix},
    \end{equation}
    where $\vecrz (Q_{12}) = N_{21}$. Then, $Q \in \mathbf{\Pi}_{p,q}$. Moreover,
    \begin{equation} \label{implication N to Q}
        Z \in \mathbb{R}^{q \times p} \text{ and }\vecrz (Z^{\top}) \in \mathcal{Z}_{qp}(N)  \Rightarrow Z \in \mathcal{Z}_{q}(Q).
    \end{equation}
\end{lemma}
\vspace{0.5em}

\begin{proof}
    We first prove that $Q \in \mathbf{\Pi}_{p,q}$. Since $N \in \mathbf{\Pi}_{1,qp}$ and $N_{22} = Q_{22} \otimes I$, it follows that $Q|Q_{22} = (N | N_{22})I \geq 0$ and $Q_{22} \leq 0$. To show $\ker Q_{22} \subseteq \ker Q_{12}$, we suppose that a vector $v\in \mathbb{R}^{q}$ is such that $Q_{22} v = 0$. This implies that $N_{22}( v \otimes I) = 0$ and hence $N_{12}( v \otimes I) = 0$. It follows from \eqref{notation vec property} that $Q_{12}v = 0$ and therefore $\ker Q_{22} \subseteq \ker Q_{12}$. We conclude that $Q \in \mathbf{\Pi}_{p,q}$.
    
    Subsequently, we prove the implication \eqref{implication N to Q}. Let $Z \in \mathbb{R}^{q \times p}$ with $\vecrz (Z^{\top}) \in \mathcal{Z}_{qp}(N)$. We have that
    \begin{equation} \label{Z in ZN}
        N|N_{22} \!+\! (\vecrz(Z^{\top}) \!+ \!N_{22}^{\dagger}N_{21})^{\!\top}N_{22}(\vecrz(Z^{\top}) \!+ \!N_{22}^{\dagger}N_{21}) \!\geq\! 0.
    \end{equation}
    Note that $N_{22}^{\dagger}N_{21} = \vecrz (Q_{12}Q_{22}^{\dagger})$. It follows from \eqref{notation tr property} that
    \begin{equation} \label{lemma general inclusion with tr}
        N|N_{22} + \trace ((Z + Q_{22}^{\dagger}Q_{21})^{\top}Q_{22}(Z + Q_{22}^{\dagger}Q_{21})) \geq 0. 
    \end{equation}
    This implies that
    \begin{equation}\label{lemma general inclusion without tr}
        (N|N_{22})I + (Z + Q_{22}^{\dagger}Q_{21})^{\top}Q_{22}(Z + Q_{22}^{\dagger}Q_{21}) \geq 0. 
    \end{equation}
    Equivalently,
    $$
    (N|N_{22})I + Q_{12} Q_{22}^{\dagger} Q_{21} + Z^{\top}Q_{21} + Q_{12}Z + Z^{\top}Q_{22}Z \geq 0. 
    $$
    Therefore, $Z \in \mathcal{Z}_{q}(Q)$, which proves the lemma.
\end{proof}

Lemma \ref{lemma general inclusion} establishes a new implication \eqref{implication N to Q}, involving the QMI solution set $\mathcal{Z}_{q}(Q)$. We note that the converse of the implication in \eqref{implication N to Q} does not hold in general. To elaborate, let $q = p = 2$ and choose
$$
N = \begin{bmatrix}
    1 & 0 \\
    0 & -I_4
\end{bmatrix}.
$$
We then compute
$$
Q = \begin{bmatrix}
    I_2 & 0\\
    0 & -I_2
\end{bmatrix}.
$$
Clearly, $I_2 \in \mathcal{Z}_{2}(Q)$ but $\vecrz (I_2^{\top}) \notin \mathcal{Z}_{4}(N)$. This shows that the converse of \eqref{implication N to Q} does not hold in general. The following lemma provides another technical result that will be used in deriving the main results of this section.

\begin{lemma} \label{lemma equal sets}
    Let $y \in \mathbb{R}^n$, $A \in \mathbb{R}^{p \times q}$ and $c \in \mathbb{R}$ with $c \geq 0$. Define
    $$
    \begin{aligned}
        \mathcal{S}_1 &:= \{By \mid B \in \mathbb{R}^{q \times n} \text{ and } \|A B\|_2 \leq c\},\\
        \mathcal{S}_2 &:= \{By \mid B \in \mathbb{R}^{q \times n} \text{ and } \|A B\|_{\mathrm{F}} \leq c\}.        
    \end{aligned}
    $$
    Then, $\mathcal{S}_1=\mathcal{S}_2$.
\end{lemma}
\begin{proof}
    When $y = 0$, the lemma trivially holds because $\mathcal{S}_1 = \mathcal{S}_2 =\{0\}$. Suppose that $y \neq 0$. Since $\|A B\|_2 \leq \|A B\|_{\mathrm{F}}$ for any $B \in \mathbb{R}^{q \times n}$, the inclusion $\mathcal{S}_2 \subseteq \mathcal{S}_1$ holds. We next prove $\mathcal{S}_1 \subseteq \mathcal{S}_2$. Let $z \in \mathcal{S}_1$. This implies that there exists $B \in \mathbb{R}^{q \times n}$ such that $z = By$ and $\|AB\|_2 \leq c$. It follows that $\|A z\|_2 = \|ABy\|_2 \leq c\|y\|_2$. Define 
    $\bar{B}:= \frac{z y^{\top}}{\|y\|_2^2}$. Then, $\bar{B} y = z$ and 
    $$
    \|A\bar{B}\|_{\mathrm{F}} = \frac{\|Az y^{\top}\|_{\mathrm{F}}}{\|y\|_2^2} = \frac{\|Az\|_2}{\|y\|_2} \leq c.
    $$
    Therefore, $z \in \mathcal{S}_2$. Since this argument holds for all $z \in \mathcal{S}_1$, we then conclude that $\mathcal{S}_1 \subseteq \mathcal{S}_2$. This proves the lemma.
\end{proof}

Building on Lemmas~\ref{lemma general inclusion} and~\ref{lemma equal sets}, we present the main results of this section.

\begin{theorem} \label{theorem main}
    Let $M \in \mathbb{S}^{p, q}$ and $N \in \mathbf{\Pi}_{1,qp}$. Assume that $N_{22} = Q_{22} \otimes I_p$ for some $Q_{22} \in \mathbb{S}^{q}$. Define $Q \in \mathbb{S}^{p, q}$ as in \eqref{definition Theta}, where $\vecrz (Q_{12}) = N_{21}$. Then, the implication \eqref{implication N to M} holds if and only if
    \begin{equation} \label{implication Q to M}
        Z \in \mathcal{Z}_{q}(Q) \Rightarrow Z \in \mathcal{Z}_{q}^+(M).
    \end{equation}
    Moreover, assume that either $M_{22} \leq 0$ or $N_{22} < 0$. Then, the implication \eqref{implication N to M} holds if and only if there exist $\alpha \geq 0$ and $\beta > 0$ such that
    \begin{equation} \label{theorem main QMI}
        M - \alpha Q \geq \begin{bmatrix}
            \beta I & 0\\
            0 & 0
        \end{bmatrix}.
    \end{equation}
\end{theorem}\vspace{0.5em}
\begin{proof}
    We now prove the first statement. The ``if" part follows directly by combining \eqref{implication N to Q} and \eqref{implication Q to M}, so we focus on the ``only if" part. Suppose that \eqref{implication N to M} holds. Let $Z \in \mathcal{Z}_{q}(Q)$. From Lemma \ref{lemma general inclusion}, it follows that $Q \in \mathbf{\Pi}_{p,q}$ and thus \eqref{lemma general inclusion without tr} holds. Equivalently,
    \begin{equation} \label{theorem main 2 norm}
        \|(-Q_{22})^{\frac{1}{2}}(Z + Q_{22}^{\dagger}Q_{21})\|_2 \leq (N|N_{22})^{\frac{1}{2}}.
    \end{equation}
    Let $x \in \mathbb{R}^p \setminus \{0\}$. Applying Lemma \ref{lemma equal sets} with $y = x$, $A = (-Q_{22})^{\frac{1}{2}}$ and $c = (N|N_{22})^{\frac{1}{2}}$, the inequality \eqref{theorem main 2 norm} implies that there exists $\bar{Z} \in \mathbb{R}^{q \times p}$ satisfying $\bar{Z}x = Zx$ and 
    \begin{equation} \label{theorem main F norm}
    \|(-Q_{22})^{\frac{1}{2}}(\bar{Z} + Q_{22}^{\dagger}Q_{21})\|_{\mathrm{F}} \leq (N|N_{22})^{\frac{1}{2}}.
    \end{equation}
    We rewrite \eqref{theorem main F norm} as
    $$
    N|N_{22} + \trace ((\bar{Z} + Q_{22}^{\dagger}Q_{21})^{\top}Q_{22}(\bar{Z} + Q_{22}^{\dagger}Q_{21})) \geq 0.
    $$
    It follows that $\vecrz (\bar{Z}^{\top}) \in \mathcal{Z}_{qp}(N)$. Hence $\bar{Z} \in \mathcal{Z}_{q}^+(M)$ and
    $$
    x^{\top}M_{11}x+x^{\top}M_{12}\bar{Z}x + x^{\top}\bar{Z}^{\top}M_{21}x + x^{\top}\bar{Z}^{\top}M_{22}\bar{Z}x > 0.
    $$
    This yields that \eqref{M Z x} holds. Since this argument holds for all nonzero $x$, we conclude that $Z \in \mathcal{Z}_{q}^+(M)$, and thus \eqref{implication Q to M} holds. 
    
    Note that $N_{22} < 0$ implies $Q_{22} < 0$. Under either condition $M_{22} \leq 0$ or $Q_{22} < 0$, the second statement follows directly from Lemma~\ref{lemma QMI}. This proves the theorem.
\end{proof}

In Theorem \ref{theorem main}, we have provided conditions under which a strict QMI is implied by a quadratic inequality in vectorized variables. The proposed conditions are based on the LMI \eqref{theorem main QMI}, which provides clear numerical benefits when compared to the condition \eqref{proposition benchmark QMI} in Proposition \ref{proposition benchmark}. However, the result is only applicable under the assumption that $N_{22} = Q_{22} \otimes I$ for some $Q_{22}$. As we will show, this assumption is satisfied in the data-driven setting under a condition on the noise model. 

Next, we present a non-strict version of Theorem \ref{theorem main}.
\begin{theorem} \label{theorem main nonstrict}
    Let $M \in \mathbb{S}^{p, q}$ and $N \in \mathbf{\Pi}_{1,qp}$. Assume that $N_{22} = Q_{22} \otimes I_p$ for some $Q_{22} \in \mathbb{S}^{q}$. Define $Q \in \mathbb{S}^{p, q}$ as in \eqref{definition Theta}, where $\vecrz (Q_{12}) = N_{21}$. Then, 
    \begin{equation}\label{implication N to M nonstrict}
        Z \in \mathbb{R}^{q \times p} \text{ and }\vecrz (Z^{\top}) \in \mathcal{Z}_{qp}(N) \Rightarrow Z \in \mathcal{Z}_{q}(M),
    \end{equation}
    if and only if
    \begin{equation} \label{implication Q to M nonstrict}
        Z \in \mathcal{Z}_{q}(Q) \Rightarrow Z \in \mathcal{Z}_{q}(M).
    \end{equation}
    Moreover, assume that $N|N_{22} > 0$. Then, the implication \eqref{implication N to M nonstrict} holds if and only if there exists an $\alpha \geq 0$ such that
    \begin{equation} \label{theorem main nonstrict QMI}
        M - \alpha Q \geq 0.
    \end{equation}
\end{theorem}\vspace{0.5em}
\begin{proof}
    The equivalence between \eqref{implication N to M nonstrict} and \eqref{implication Q to M nonstrict} follows directly by adapting the proof of Theorem \ref{theorem main}. It then suffices to show that under the assumption that $N|N_{22}>0$, \eqref{implication Q to M nonstrict} holds if and only if there exists an $\alpha \geq 0$ satisfying \eqref{theorem main nonstrict QMI}. Since $Q|Q_{22} = N|N_{22}I$ and $N|N_{22}>0$, $Q$ has $p$ positive eigenvalues. By \cite[Thm.~4.7]{DDCQMI}, we conclude that \eqref{implication Q to M nonstrict} holds if and only if there exists an $\alpha \geq 0$ satisfying \eqref{theorem main nonstrict QMI}. This proves the theorem.
\end{proof}

\section{Data-driven control design} \label{section Data-driven control design}
In what follows, we apply the results in Section \ref{section Methodology} to data-driven $M$-control design. First, Section \ref{section The general M-control problem} studies the general problem of $M$-control. The results are then specialized to data-driven quadratic stabilization, followed by the incorporation of $\mathcal{H}_2$ and $\mathcal{H}_{\infty}$ performance specifications. Finally, Section \ref{section Discussion on the instantaneous noise bound} investigates the advantages of our $M$-control design approaches over the existing approaches when the noise is instantaneously bounded. Recall that $\Phi \in \mathbf{\Pi}_{1,nT}$. In the rest of the paper, we impose the following additional assumption on $\Phi$.

\begin{assumption}
    The matrix $\Phi_{22} = \phi \otimes I_n$ for some $\phi \in \mathbb{S}^T$.
\end{assumption}

\subsection{The general $\mathit{M}$-control problem} \label{section The general M-control problem}
We now address Problem~\ref{problem 1}, namely, we derive conditions under which the data $\mathcal{D}$ are informative for $M$-control with respect to $\mathcal{W}_{\mathrm{F}}$. Recall that $M: \Theta \rightarrow \mathbb{S}^{n,n+m}$. Define $N$ as in \eqref{definition N}, i.e.,
$$
N = 
    \begin{bmatrix}
     1 & 0 \\
    \vecrz(X_+)\! &\! -\begin{bmatrix}
    X_-\! \\ U_-\!
\end{bmatrix}^{\!\!\top}\!\! \otimes I
\end{bmatrix}^{\!\!\top}\!\Phi\!
\begin{bmatrix}
     1 & 0 \\
    \!\vecrz(X_+)\! &\! -\begin{bmatrix}
    X_-\! \\ U_-\!
\end{bmatrix}^{\!\!\top}\!\! \otimes I
\end{bmatrix}\!.
$$
It follows from \eqref{vec A B} that $(A,B) \in \Sigma_{\mathcal{D}}(\mathcal{W}_{\mathrm{F}})$ if and only if $\vecrz(\begin{bmatrix}
    A & B
\end{bmatrix}) \in \mathcal{Z}_{(n+m)n}(N)$. The problem then amounts to find conditions under which there exists a $\theta \in \Theta$ such that 
\begin{equation} \label{implication N to M Mcontrol}
    \begin{aligned}
    Z  \in  \mathbb{R}^{(n+m) \times n} \text{ and }\vecrz (Z^{\top})  \in  \mathcal{Z}_{(n+m)n}(N) &\\
    \Rightarrow   Z  \in  \mathcal{Z}_{n+m}^+&(M(\theta)).
\end{aligned}
\end{equation}
We will apply Theorem \ref{theorem main} to derive a solution. Note that $M_{22}(\theta) \leq 0$ for all $\theta \in \Theta$ and $N \in \mathbf{\Pi}_{1,n(n+m)}$. Define 
\begin{equation}\label{definition Theta 22}
    Q_{22}:= \begin{bmatrix}
    X_- \\
    U_-
\end{bmatrix} \phi \begin{bmatrix}
    X_-\\
    U_-
\end{bmatrix}^{\top}.
\end{equation}
Since $\Phi_{22} = \phi \otimes I_n$, we have $N_{22} = Q_{22} \otimes I$. Therefore, the second statement in Theorem \ref{theorem main} is applicable. Specifically, define $Q \in \mathbb{R}^{n + (n+m)}$ as 
\begin{equation} \label{definition Q ddc}
    Q := \begin{bmatrix}
        (N | N_{22})I + Q_{12} Q_{22}^{\dagger} Q_{21} & Q_{12} \\
        Q_{21} & Q_{22}
    \end{bmatrix},
\end{equation}
where $\vecrz (Q_{12}) = N_{21}$. By Theorem~\ref{theorem main}, the implication \eqref{implication N to M Mcontrol} holds if and only if there exist scalars $\alpha \geq 0$ and $\beta > 0$ such that
    \begin{equation} \label{theorem M control LMI}
        M(\theta)- \alpha Q \geq \begin{bmatrix}
            \beta I & 0\\
            0 & 0
        \end{bmatrix}.
   \end{equation}
Building on this, the design conditions are presented in the following theorem.
\begin{theorem} \label{theorem M control}
    Let $\Theta$ be a set of control parameters and $M: \Theta \rightarrow \mathbb{S}^{n,n+m}$ be a mapping with $M_{22}(\theta) \leq 0$ for all $\theta \in \Theta$. Let $Q \in \mathbb{S}^{n,n+m}$ be defined as in \eqref{definition Q ddc}. Then, the data $\mathcal{D}$ are informative for $M$-control with respect to $\mathcal{W}_{\mathrm{F}}$ if and only if there exist a $\theta \in \Theta$ and scalars $\alpha \geq 0$ and $\beta > 0$ such that \eqref{theorem M control LMI} holds.
\end{theorem}
\begin{proof}
    To prove the ``if" statement, we suppose that there exist $\theta$, $\alpha$ and $\beta$ such that \eqref{theorem M control LMI} holds. Let $(A,B)\in \Sigma_{\mathcal{D}}(\mathcal{W}_{\mathrm{F}})$. This implies that $\vecrz (\begin{bmatrix}
        A & B
    \end{bmatrix}) \in \mathcal{Z}_{(n+m)n}(N)$. It follows from Lemma~\ref{lemma general inclusion} that $\begin{bmatrix}
        A & B
    \end{bmatrix}^{\top} \in \mathcal{Z}_{n+m}(Q)$. Multiply \eqref{theorem M control LMI} from left by $\begin{bmatrix}
        I & A & B
    \end{bmatrix}$ and from right by its transpose. This yields \eqref{definition 1 M theta A B}. Since this argument holds for all $(A,B)\in \Sigma_{\mathcal{D}}(\mathcal{W}_{\mathrm{F}})$, we conclude that the data $\mathcal{D}$ are informative for $M$-control with respect to $\mathcal{W}_{\mathrm{F}}$.

    We next prove the ``only if" part. Suppose that the data $\mathcal{D}$ are informative for $M$-control with respect to $\mathcal{W}_{\mathrm{F}}$. This implies that there exists a $\theta \in \Theta$ such that \eqref{definition 1 M theta A B} holds for all $\vecrz (\begin{bmatrix}
        A & B
    \end{bmatrix}) \in \mathcal{Z}_{(n+m)n}(N)$. By Theorem \ref{theorem main}, there exist scalars $\alpha \geq 0$ and $\beta > 0$ satisfying \eqref{theorem M control LMI}. This proves the theorem.
\end{proof}

\subsection{Data-driven quadratic stabilization}
We next apply the results of Section \ref{section The general M-control problem} to data-driven quadratic stabilization. In particular, we specify $(\Theta, M) = (\Theta_{\mathrm{stab}}, M_{\mathrm{stab}})$, where $\Theta_{\mathrm{stab}}$ and $M_{\mathrm{stab}}$ are defined as in \eqref{definition Theta stab} and \eqref{definition M stab}, respectively. Since $(P,K) \in \Theta_{\mathrm{stab}}$ implies $P>0$, the $(2,2)$-block of $M_{\mathrm{stab}}(P,K)$, i.e., $-\begin{bmatrix}
    I & K^{\top}
\end{bmatrix}^{\top} P \begin{bmatrix}
    I & K^{\top}
\end{bmatrix}$, is negative semidefinite for all $(P,K) \in \Theta_{\mathrm{stab}}$. By Theorem \ref{theorem M control}, the data $\mathcal{D}$ are informative for $M_{\mathrm{stab}}$-control with respect to $\mathcal{W}_{\mathrm{F}}$ if and only if there exist a $(P,K) \in \Theta_{\mathrm{stab}}$ and scalars $\alpha \geq 0$ and $\beta > 0$ such that
\begin{equation} \label{QMI M stab Q}
    M_{\mathrm{stab}}(P,K) - \alpha Q \geq \begin{bmatrix}
            \beta I & 0\\
            0 & 0
        \end{bmatrix}.
\end{equation}
In \cite{DDCmatrixSlemma}, data informativity for $M_{\mathrm{stab}}$-control with respect to $\mathcal{W}_{\mathrm{qmi}}(\Psi)$ was studied. In that paper, this notion of informativity was called \emph{data informativity for quadratic stabilization}. We stress that in this paper, we focus on conditions for informativity with respect to the new noise model $\mathcal{W}_{\mathrm{F}}$. Our data-driven stabilization method is presented in the following theorem.
\begin{theorem} \label{theorem stab}
    Let $\Theta_{\mathrm{stab}}$ and $M_{\mathrm{stab}}\! : \! \Theta_{\mathrm{stab}}\!\rightarrow\!\mathbb{S}^{n,n+m}$ be defined as in \eqref{definition Theta stab} and \eqref{definition M stab}, respectively. Let $Q \in \mathbb{S}^{n,n+m}$ be defined as in \eqref{definition Q ddc}. The data $\mathcal{D}$ are informative for $M_{\mathrm{stab}}$-control with respect to $\mathcal{W}_{\mathrm{F}}$ if and only if there exist an $n \times n$ matrix $P > 0$, an $L \in \mathbb{R}^{m \times n}$ and a scalar $\beta>0$, such that 
    \begin{equation} \label{theorem stab QMI}
    \begin{bmatrix}
        P - \beta I & 0 & 0 & 0\\
        0 & -P & -L^{\top} & 0\\
        0 & -L & 0 & L\\
        0 & 0 & L^{\top} & P
    \end{bmatrix} - \begin{bmatrix}
        Q & 0\\
        0 & 0
    \end{bmatrix} \geq 0.
    \end{equation}
    Moreover, if $P$ and $L$ satisfy \eqref{theorem stab QMI}, then $K := L P^{-1}$ is a stabilizing feedback gain for all $(A,B) \in \Sigma_{\mathcal{D}}(\mathcal{W}_{\mathrm{F}})$. 
\end{theorem}
\begin{proof}
    By Theorem \ref{theorem M control}, it suffices to show that \eqref{QMI M stab Q} is feasible if and only if \eqref{theorem stab QMI} is feasible. To this end, we first prove the ``if" part. Suppose that there exist $P$, $L$ and $\beta$ satisfying \eqref{theorem stab QMI}. Define $K := LP^{-1}$. Then, $(P,K) \in \Theta_{\mathrm{stab}}$. By computing the Schur complement of \eqref{theorem stab QMI} with respect to its fourth diagonal block, we obtain \eqref{QMI M stab Q} with $\alpha = 1$. Moreover, $K = L P^{-1}$ is a stabilizing feedback gain for all $(A,B) \in \Sigma_{\mathcal{D}}(\mathcal{W}_{\mathrm{F}})$.

    We next prove the ``only if" part. Suppose that there exist $(P,K) \in \Theta_{\mathrm{stab}}$ and scalars $\alpha \geq 0$ and $\beta > 0$ such that \eqref{QMI M stab Q} holds. Define $L:=KP$. Using a Schur complement argument, we have that
    $$
    \begin{bmatrix}
        P - \beta I & 0 & 0 & 0\\
        0 & -P & -L^{\top} & 0\\
        0 & -L & 0 & L\\
        0 & 0 & L^{\top} & P
    \end{bmatrix} - \alpha \begin{bmatrix}
        Q & 0\\
        0 & 0
    \end{bmatrix} \geq 0.
    $$
    Note that the second diagonal block is $-P - \alpha X_- \phi X_-^{\top}$. Since $\Phi \in \mathbf{\Pi}_{1,nT}$, it follows that $\phi \leq 0$. Therefore, $P>0$ implies $\alpha >0$. By scaling $P$, $L$ and $\beta$ by $\frac{1}{\alpha}$, we conclude that \eqref{theorem stab QMI} is feasible. This proves the theorem.
\end{proof}

\subsection{Inclusion of $\mathcal{H}_2$ and $\mathcal{H}_{\infty}$ performance specifications} \label{section Inclusion of performance specifications}
In the following, we investigate data-driven $\mathcal{H}_2$ and $\mathcal{H}_{\infty}$ control problems. To this end, we associate with the unknown system \eqref{eq real system} a performance output
\begin{equation} \label{eq output performance}
    y(t) = Cx(t)+Du(t),
\end{equation}
where $y \in \mathbb{R}^p$, and $C \in \mathbb{R}^{p \times n}$ and $D\in \mathbb{R}^{p \times m}$ are given matrices that specify the performance. We denote by 
$$
G(z):=(C+DK)(zI - (A+BK))^{-1},
$$
the transfer matrix from $w$ to $y$ of the closed-loop system obtained by interconnecting any $(A,B) \in \Sigma_{\mathcal{D}}(\mathcal{W}_{\mathrm{F}})$ with the controller $u = Kx$. 

We first study the $\mathcal{H}_2$ control problem and denote the $\mathcal{H}_2$ norm of $G(z)$ by $\|G(z)\|_{\mathcal{H}_2}$. Let $\gamma > 0$. By \cite[Prop.~8.1]{Henkbook}, it follows that $A+BK$ is stable and $\|G(z)\|_{\mathcal{H}_2} < \gamma$ if and only if there exists a symmetric matrix $P>0$ such that
\begin{subequations}
\begin{align}
P &> (A +\! BK)^{\!\top} P (A +\! BK) + (C +\! DK)^{\!\top}(C +\! DK) \label{H2 a}\\
\!\trace(P) &< \gamma^2 \label{H2 b}
\end{align}
\end{subequations}
Following \cite{DDCmatrixSlemma}, we let
\begin{equation} \label{definition Y L CYL}
    Y := P^{-1},\ L := KY \text{ and } C_{Y,L} := C Y + DL.
\end{equation}
Then, $P > 0$ and \eqref{H2 a} are equivalent to 
\begin{equation} \label{Y CYL}
    Y - C_{Y,L}^{\top}C_{Y,L} >0,
\end{equation}
and
\begin{equation} \label{M H2 A B}
\begin{bmatrix}
     I \\ A^{\top}\! \\ B^{\top}\!
\end{bmatrix}^{\!\!\top}\!\!\begin{bmatrix}
        Y & 0 \\
        0 & \!\!-\begin{bmatrix}
            Y\\L
        \end{bmatrix}\!\! (Y \!- C_{Y,L}^{\top}C_{Y,L})^{-1}\!\begin{bmatrix}
            Y\\L
        \end{bmatrix}^{\!\!\top}\!
    \end{bmatrix}\!\!\!\begin{bmatrix}
     I \\ A^{\top}\! \\ B^{\top}\!
\end{bmatrix}\! >0.
\end{equation}
Define
\begin{equation} \label{definition Theta H2}
\begin{aligned}
    \Theta_{\mathcal{H}_{2}}\!\!:=\! \{(Y,L)\!\in \mathbb{S}^{n} \!\times \!\mathbb{R}^{n \times m} \mid \trace(Y^{-1}) \! < \! \gamma^2 \text{ and } \eqref{Y CYL} \text{ holds}\}\!,
\end{aligned}
\end{equation}
and 
\begin{equation} \label{definition M H2}
    M_{\mathcal{H}_2}(Y,L) \!:=\! \begin{bmatrix}
        Y & 0 \\
        0 & -\begin{bmatrix}
            Y\\L
        \end{bmatrix}\! (Y \!- C_{Y,L}^{\top}C_{Y,L})^{-1}\begin{bmatrix}
            Y\\L
        \end{bmatrix}^{\!\top}
    \end{bmatrix}\!.
\end{equation}
Since $(Y,L) \in \Theta_{\mathcal{H}_2}$ implies \eqref{Y CYL}, the $(2,2)$-block of $M_{\mathcal{H}_2}(Y,L)$ is negative semidefinite for all $(Y,L) \in \Theta_{\mathcal{H}_2}$. By applying Theorem \ref{theorem M control} with $(\Theta,M) = (\Theta_{\mathcal{H}_{2}}, M_{\mathcal{H}_2})$, the data $\mathcal{D}$ are informative for $M_{\mathcal{H}_2}$-control with respect to $\mathcal{W}_{\mathrm{F}}$ if and only if there exist a $(Y,L) \in \Theta_{\mathcal{H}_{2}}$ and scalars $\alpha \geq 0 $ and $\beta >0$ satisfying 
\begin{equation} \label{theorem H2 M Theta}
    M_{\mathcal{H}_2}(Y,L) - \alpha Q \geq \begin{bmatrix}
        \beta I & 0\\
        0 & 0
    \end{bmatrix}.
\end{equation}
Note that data informativity for $M_{\mathcal{H}_2}$-control with respect to $\mathcal{W}_{\mathrm{qmi}}(\Psi)$ was studied in \cite{DDCmatrixSlemma}, where this notion was called data informativity for $\mathcal{H}_2$ control with performance $\gamma$. In this paper, we focus on the new noise model $\mathcal{W}_{\mathrm{F}}$. Our data-driven $\mathcal{H}_2$ control method is summarized as follows.

\begin{theorem} \label{theorem H2}
    Let $\Theta_{\mathcal{H}_2}$ and $M_{\mathcal{H}_2}: \Theta_{\mathcal{H}_2} \rightarrow \mathbb{S}^{n,n+m}$ be defined as in \eqref{definition Theta H2} and \eqref{definition M H2}, respectively. Let $Q \in \mathbb{S}^{n,n+m}$ be defined as in \eqref{definition Q ddc}. The data $\mathcal{D}$ are informative for $M_{\mathcal{H}_2}$-control with respect to $\mathcal{W}_{\mathrm{F}}$ if and only if there exist matrices $Y, Z \in \mathbb{S}^n$ and $L \in \mathbb{R}^{m \times n}$ and scalars $\alpha \geq 0$ and $\beta > 0$ such that 
    \begin{equation} \label{theorem H2 QMI}
        \begin{bmatrix} 
            Y - \beta I & 0 & 0 & 0 & 0 \\
            0 & 0 & 0 & Y & 0\\
            0 & 0 & 0 & L & 0\\
            0 & Y & L^{\top} & Y & C_{Y,L}^{\top}\\
            0 & 0 & 0 & C_{Y,L} & I\\
        \end{bmatrix} - \alpha \begin{bmatrix}
            Q & 0 \\
            0 & 0
        \end{bmatrix} \geq 0,
    \end{equation}
    \begin{equation} \label{theorem H2 others}
        \begin{bmatrix}
            Y & C_{Y,L}^{\top}\\
            C_{Y,L} & I
        \end{bmatrix} > 0,\ 
        \begin{bmatrix}
            Z & I\\
            I & Y
        \end{bmatrix}\geq 0, \ 
        \trace(Z) < \gamma^2.
    \end{equation}
    Moreover, if $Y$ and $L$ satisfy
    \eqref{theorem H2 QMI} and \eqref{theorem H2 others}, then $K:= LY^{-1}$ is such that $A+BK$ is stable and $\|G(z)\|_{\mathcal{H}_2} < \gamma$ for all $(A,B) \in \Sigma_{\mathcal{D}}(\mathcal{W}_{\mathrm{F}})$.
\end{theorem}
\begin{proof}
    Following the proof of Theorem \ref{theorem stab}, we show that \eqref{theorem H2 M Theta} is feasible if and only if \eqref{theorem H2 QMI} and \eqref{theorem H2 others} are feasible. To this end, we first prove the ``if" part. Suppose that there exist $Y$, $Z$, $L$, $\alpha$ and $\beta$ satisfying \eqref{theorem H2 QMI} and \eqref{theorem H2 others}. Using a Schur complement argument, the first inequality in \eqref{theorem H2 others} implies $Y > 0$ and \eqref{Y CYL}. Similarly, the second inequality implies $Z \geq Y^{-1}$. Together with the third inequality, it follows that $\trace(Y^{-1}) \leq \trace(Z) < \gamma^2$. Therefore, $(Y,L) \in \Theta_{\mathcal{H}_{2}}$. Next, we compute the Schur complement of \eqref{theorem H2 QMI} with respect to its fifth diagonal block and then compute the Schur complement of the resulting matrix inequality with respect to its fourth diagonal block. This yields that \eqref{theorem H2 M Theta} is feasible. Let $P := Y^{-1}$ and $K:= LY^{-1}$. It follows that \eqref{H2 a} holds for all $(A,B) \in \Sigma_{\mathcal{D}}(\mathcal{W}_{\mathrm{F}})$ and \eqref{H2 b} is satisfied. We then conclude that with $K = LY^{-1}$, $A+BK$ is stable and $\|G(z)\|_{\mathcal{H}_2} < \gamma$ for all $(A,B) \in \Sigma_{\mathcal{D}}(\mathcal{W}_{\mathrm{F}})$. 

    To prove the ``only if" part, we suppose that there exist $Y$, $L$, $\alpha$ and $\beta$ satisfying \eqref{theorem H2 M Theta}. Let $Z := Y^{-1}$. It follows from $(Y,L) \in \Theta_{\mathcal{H}_{2}}$ that the inequalities in \eqref{theorem H2 others} are satisfied. Using a Schur complement argument, \eqref{theorem H2 M Theta} implies \eqref{theorem H2 QMI}. This proves the theorem.
\end{proof}

We next study the $\mathcal{H}_{\infty}$ control problem and denote the $\mathcal{H}_{\infty}$ norm of $G(z)$ by $\|G(z)\|_{\mathcal{H}_{\infty}}$. Let $\gamma > 0$. By adapting \cite[Lem.~8.12]{Henkbook}, the matrix $A+BK$ is stable and $\|G(z)\|_{\mathcal{H}_{\infty}} < \gamma$ if and only if there exists a symmetric matrix $P>0$ such that
\begin{equation} \label{H infty}
\begin{bmatrix}
    P - (C +\! DK)^{\!\top}(C +\! DK) & (A +\! BK)^{\!\top}\\
    (A +\! BK) & P^{-1}- \frac{1}{\gamma^2} I
\end{bmatrix} > 0. 
\end{equation}
Similar to the discussion on data-driven $\mathcal{H}_2$ control, define $Y$, $L$ and $C_{Y,L}$ as in \eqref{definition Y L CYL}. Then, $P > 0$ and \eqref{H infty} are equivalent to \eqref{Y CYL} and
$$
\begin{bmatrix}
 I \\ A^{\!\top} \\ B^{\!\top}
\end{bmatrix}^{\!\!\top}\!\!\begin{bmatrix}
        Y \!-\! \frac{1}{\gamma^2}I & 0 \\
        0 & \!\!-\!\begin{bmatrix}
            Y\\L
        \end{bmatrix}\!\! (Y \!- C_{Y,L}^{\top}C_{Y,L})^{-1}\!\begin{bmatrix}
            Y\\L
        \end{bmatrix}^{\!\!\top}
    \end{bmatrix}\!\!\!\begin{bmatrix}
 I \\ A^{\!\top} \\ B^{\!\top}
\end{bmatrix} \!>0.
$$
Define
\begin{equation} \label{definition Theta Hinf}
    \Theta_{\mathcal{H}_{\infty}}\! :=\! \{(Y,L) \in \mathbb{S}^{n} \times \mathbb{R}^{n \times m} \mid \eqref{Y CYL} \text{ holds}\},
\end{equation}
and 
\begin{equation} \label{definition M Hinf}
    M_{\mathcal{H}_{\infty}}(Y,L)\! := \! \! \begin{bmatrix}
        Y \!-\! \frac{1}{\gamma^2}I & 0 \\
        0 & \!\!-\!\begin{bmatrix}
            Y\\L
        \end{bmatrix}\! (Y \!- \!C_{Y,L}^{\top}C_{Y,L})^{-1}\!\begin{bmatrix}
            Y\\L
        \end{bmatrix}^{\!\!\top}\!
    \end{bmatrix}\!\!.
\end{equation} 
Since $(Y,L) \in \Theta_{\mathcal{H}_{\infty}}$ implies \eqref{Y CYL}, the $(2,2)$ block of $M_{\mathcal{H}_{\infty}}(Y,L)$ is negative semidefinite for all $(Y,L) \in \Theta_{\mathcal{H}_{\infty}}$. Theorem \ref{theorem M control} can then be applied with $(\Theta,M) = (\Theta_{\mathcal{H}_{\infty}}, M_{\mathcal{H}_{\infty}})$. In \cite{DDCmatrixSlemma}, data informativity for $M_{\mathcal{H}_{\infty}}$-control with respect to $\mathcal{W}_{\mathrm{qmi}}(\Psi)$ was studied, where the notion was called data informativity for $\mathcal{H}_{\infty}$ control with performance~$\gamma$. In this paper, we consider the new noise model $\mathcal{W}_{\mathrm{F}}$. The following theorem presents our data-driven $\mathcal{H}_{\infty}$ control method.

\begin{theorem} \label{theorem H infty}
    Let $\Theta_{\mathcal{H}_{\infty}}$ and $M_{\mathcal{H}_{\infty}}: \Theta_{\mathcal{H}_{\infty}} \rightarrow \mathbb{S}^{n,n+m}$ be defined as in \eqref{definition Theta Hinf} and \eqref{definition M Hinf}, respectively. Let $Q \in \mathbb{S}^{n,n+m}$ be defined as in \eqref{definition Q ddc}. The data $\mathcal{D}$ are informative for $M_{\mathcal{H}_{\infty}}$-control with respect to $\mathcal{W}_{\mathrm{F}}$ if and only if there exist matrices $Y, Z \in \mathbb{S}^n$ and $L \in \mathbb{R}^{m \times n}$ and scalars $\alpha \geq 0$ and $\beta > 0$ such that
    \begin{equation} \label{H infty QMI}
        \begin{bmatrix} 
            Y \!-\! \frac{1}{\gamma^2}I \!- \!\beta I & 0 & 0 & 0 & 0 \\
            0 & 0 & 0 & Y & 0\\
            0 & 0 & 0 & L & 0\\
            0 & Y & L^{\!\top} & Y & C_{Y,L}^{\top}\\
            0 & 0 & 0 &\! C_{Y,L} \! & I\\
        \end{bmatrix} - \alpha \begin{bmatrix}
            Q & 0 \\
            0 & 0
        \end{bmatrix} \geq 0,
    \end{equation}
    \begin{equation} \label{H infty others}
        \begin{bmatrix}
            Y & C_{Y,L}^{\top}\\
            C_{Y,L} & I
        \end{bmatrix} > 0.
    \end{equation}
    Moreover, if $Y$ and $L$ satisfy
    \eqref{H infty QMI} and \eqref{H infty others}, then $K\!:= \! LY^{-1}$ is such that $A\!+\!BK$ is stable and $\|G(z)\|_{\mathcal{H}_{\infty}} \!< \!\gamma$ for all $(A,B) \!\in \!\Sigma_{\mathcal{D}}(\mathcal{W}_{\mathrm{F}})$.
\end{theorem}

The proof of Theorem~\ref{theorem H infty} follows by adapting the proof of Theorem~\ref{theorem H2} and is therefore omitted here.

\subsection{Discussion on the instantaneous noise bound} \label{section Discussion on the instantaneous noise bound}
In this section, we consider the case that noise is instantaneously bounded, i.e., $W_- \in \mathcal{W}_{\mathrm{ib}}(\epsilon)$, and investigate the advantages of the proposed design method over the one based on $\mathcal{W}_{\mathrm{qmi}}(\Psi_{\epsilon})$. By Lemma \ref{lemma W set inclusion}, the inclusion \eqref{W inclusion}, i.e.,
$$
\mathcal{W}_{\mathrm{ib}}(\epsilon) \subseteq \mathcal{W}_{\mathrm{F}}(\Phi_{\epsilon}) \subseteq \mathcal{W}_{\mathrm{qmi}}(\Psi_{\epsilon}),
$$
holds, and therefore the set $\mathcal{W}_{\mathrm{F}}(\Phi_{\epsilon})$ is a less conservative overapproximation of $\mathcal{W}_{\mathrm{ib}}(\epsilon)$ than $\mathcal{W}_{\mathrm{qmi}}(\Psi_{\epsilon})$. 

We next demonstrate how this reduced conservatism translates into improved control design conditions. It is shown in Theorem~\ref{theorem M control} that the $\mathcal{W}_{\mathrm{F}}(\Phi_{\epsilon})$-based approach requires feasibility of \eqref{theorem M control LMI}, whereas statement \ref{proposition statement a QMI set} in Proposition \ref{proposition preliminaries} shows that $\mathcal{W}_{\mathrm{qmi}}(\Psi_{\epsilon})$ yields conditions based on \eqref{proposition statement a QMI set LMI}. We observe that \eqref{theorem M control LMI} has a form similar to \eqref{proposition statement a QMI set LMI}, thereby retaining favorable numerical properties. Moreover, $\mathcal{W}_{\mathrm{F}}(\Phi_{\epsilon})$-based approach can be less conservative. Specifically, if there exist $\theta$, $\alpha$ and $\beta$ satisfying \eqref{proposition statement a QMI set LMI}, then \eqref{theorem M control LMI} also holds with the same $\theta$, $\alpha$ and $\beta$. To see this, suppose that there exist a $\theta \in \Theta$ and scalars $\alpha \geq 0$ and $\beta>0$ satisfying \eqref{proposition statement a QMI set LMI} with $\Psi = \Psi_{\epsilon}$. It can be verified that the matrix $Q$ in \eqref{theorem M control LMI} satisfies 
$$
Q_{12} = Q_{21}^{\top} = X_+ \begin{bmatrix}
    X_- \\
    U_-
\end{bmatrix}^{\top}\text{ and }Q_{22} = -\begin{bmatrix}
    X_- \\
    U_-
\end{bmatrix} \begin{bmatrix}
    X_- \\
    U_-
\end{bmatrix}^{\top}.
$$ 
Define
$$
\Delta := X_+\left(I - \begin{bmatrix}
    X_-\\U_-
\end{bmatrix}^{\top}\left(\begin{bmatrix}
    X_-\\U_-
\end{bmatrix}\begin{bmatrix}
    X_-\\U_-
\end{bmatrix}^{\top}\right)^{\dagger}\begin{bmatrix}
    X_-\\U_-
\end{bmatrix}\right )X_+^{\top}.
$$
Then, 
$$
\begin{bmatrix}
    I \!&\! X_+ \\
    0 \!&\! -X_- \\
    0 \!&\! -U_-
\end{bmatrix}\Psi_{\epsilon} \begin{bmatrix}
    I \!&\! X_+ \\
    0 \!&\! -X_- \\
    0 \!&\! -U_-
\end{bmatrix}^{\!\!\top} = Q + \begin{bmatrix}
    \trace (\Delta)I - \Delta & 0 \\
    0 & 0
\end{bmatrix}\!,
$$
Since $\Delta \geq 0$, we obtain $\trace (\Delta)I - \Delta \geq 0$. This implies that \eqref{theorem M control LMI} holds. Therefore, feasibility of \eqref{proposition statement a QMI set LMI} with $\Psi = \Psi_{\epsilon}$ implies feasibility of \eqref{theorem M control LMI}. We conclude that, for instantaneously bounded noise, the $\mathcal{W}_{\mathrm{F}}(\Phi_{\epsilon})$-based approach is typically less conservative than the $\mathcal{W}_{\mathrm{qmi}}(\Psi_{\epsilon})$-based approach. In Sections \ref{section A numerical example on quadratic stabilization} and \ref{section a pendulum cart system}, we further illustrate this with numerical examples. 

\section{Data-driven analysis} \label{section data driven analysis}
In this section, we extend the proposed control design framework to data-driven analysis problems. In particular, we study quadratic stabilizability in Section \ref{section Quadratic stabilizability} and then dissipativity analysis in Section \ref{section Dissipativity analysis}.

\subsection{Quadratic stabilizability} \label{section Quadratic stabilizability}
In the following, we aim to derive conditions under which all systems in $\Sigma_{\mathcal{D}}(\mathcal{W}_{\mathrm{F}})$ are quadratic stabilizable. By adapting \cite[Def.~3.13~(d)]{Henkbook}, we introduce the following definition of informativity.
\begin{definition} \label{definition quadratic stabilizability}
    The data $\mathcal{D}$ are called \textit{informative for quadratic stabilizability} if there exists a real matrix $P>0$ such that 
    \begin{equation} \label{definition quadratic stabilizability QMI}
        \begin{bmatrix} I \\ A^\top \\ B^\top \end{bmatrix}^\top \begin{bmatrix}
            P & 0 & 0\\
            0 & -P & 0\\
            0 & 0 & I
        \end{bmatrix}
        \begin{bmatrix} I \\ A^\top \\ B^\top \end{bmatrix} > 0,
    \end{equation}
    for all $(A,B) \in \Sigma_{\mathcal{D}}(\mathcal{W}_{\mathrm{F}})$.
\end{definition}

Definition \ref{definition quadratic stabilizability} requires a common $P$ for all systems in $\Sigma_{\mathcal{D}}(\mathcal{W}_{\mathrm{F}})$. Note that these systems do not necessarily share a common stabilizing gain $K$. Therefore, this notion is in general weaker than informativity for $M_{\mathrm{stab}}$-control with respect to $\mathcal{W}_{\mathrm{F}}$. The goal of this section is to find conditions under which the data $\mathcal{D}$ are informative for quadratic stabilizability.

If $P$ is fixed, then \eqref{definition quadratic stabilizability QMI} is a strict QMI in $\begin{bmatrix}
    A & B
\end{bmatrix}$. Following the steps of control design, we apply Theorem \ref{theorem main} to derive conditions. This leads to the following theorem.

\begin{theorem} \label{theorem quadratic stabilizability}
    Let $Q \in \mathbb{S}^{n,n+m}$ be defined as in \eqref{definition Q ddc}. Assume that $\phi < 0$ and $\begin{bmatrix}
    X_-\\U_-
\end{bmatrix}$ has full row rank. The data $\mathcal{D}$ are informative for quadratic stabilizability if and only if exists a matrix $P > 0$ and scalars $\alpha \geq 0$ and $\beta>0$ such that 
    \begin{equation} \label{theorem quadratic stabilizability LMI}
        \begin{bmatrix}
            P - \beta I & 0 & 0\\
            0 & -P & 0\\
            0 & 0 & I
        \end{bmatrix} - \alpha Q \geq 0.
    \end{equation}
\end{theorem}
\vspace{1em}
\begin{proof}  
The proof of Theorem~\ref{theorem quadratic stabilizability} follows the same steps as the discussion in Section~\ref{section The general M-control problem}. We present only the differences. In this case, define $M \in \mathbb{S}^{n,n+m}$ as
$$
M:=\begin{bmatrix}
            P - \beta I & 0 & 0\\
            0 & -P & 0\\
            0 & 0 & I
        \end{bmatrix},
$$
and define $N \in \mathbb{S}^{1, n(n+m)}$ as in \eqref{definition N}. Since $\phi < 0$ and $\begin{bmatrix}
    X_-\\U_-
\end{bmatrix}$ has full row rank, we have $N_{22} <0$. The results then follows by applying Theorem~\ref{theorem main}.
\end{proof}

Next, we simplify the LMI condition \eqref{theorem quadratic stabilizability LMI} in the following corollary

\begin{corollary}
    Let $Q\in \mathbb{S}^{n,n+m}$ be defined as in \eqref{definition Q ddc}. The data $\mathcal{D}$ are informative for quadratic stabilizability if and only if exist a matrix $P > 0$ and a scalar $\beta>0$ such that 
    \begin{equation} \label{corollary quadratic stabilizability LMI}
        \begin{bmatrix}
            P - \beta I & 0\\
            0 & -P
        \end{bmatrix} - \begin{bmatrix}
            I_{2n} & 0
        \end{bmatrix}Q\begin{bmatrix}
            I_{2n} & 0
        \end{bmatrix}^{\top} \geq 0.
    \end{equation}
\end{corollary}
\vspace{1em}
\begin{proof}
    It suffices to prove that \eqref{theorem quadratic stabilizability LMI} is feasible if and only if \eqref{corollary quadratic stabilizability LMI} is feasible. We first prove the ``only if" part and suppose that there exist $P$, $\alpha$ and $\beta$ satisfying \eqref{theorem quadratic stabilizability LMI}. Since $P>0$, it follows that $\alpha > 0$. By dividing \eqref{theorem quadratic stabilizability LMI} by $\alpha$, we have
    $$
     \begin{bmatrix}
            \frac{1}{\alpha}P -  \frac{\beta}{\alpha}I & 0\\
            0 & -\frac{1}{\alpha}P
        \end{bmatrix} - \begin{bmatrix}
            I_{2n} & 0
        \end{bmatrix}Q\begin{bmatrix}
            I_{2n} & 0
        \end{bmatrix}^{\top} \geq 0.
    $$
    This implies that \eqref{corollary quadratic stabilizability LMI} is feasible.

    Next, we prove the ``if" part and suppose that there exist $P$ and $\beta$ satisfying \eqref{corollary quadratic stabilizability LMI}. Let scalars $c_1$ and $c_2$ be such that $0 < c_1 < 1$, $c_2 > 0$ and $c_1 P + c_2 I < \beta I$. Then,
    $$
    \begin{bmatrix}
            (1-c_1)P - c_2 I & 0\\
            0 & \!\!-(1-c_1)P
        \end{bmatrix} - \begin{bmatrix}
            I_{2n} \!&\! 0
        \end{bmatrix}Q\begin{bmatrix}
            I_{2n} \!&\! 0
        \end{bmatrix}^{\!\top} \!> 0.
    $$
    Moreover, there exists a sufficiently large $c_3 >0$ such that 
    \begin{equation} \label{corollary stab in proof}
        \begin{bmatrix}
            (1-c_1)P - c_2 I & 0 & 0\\
            0 & -(1-c_1)P & 0\\
            0 & 0 & c_3 I
        \end{bmatrix} - Q > 0. 
    \end{equation}
    By dividing \eqref{corollary stab in proof} by $c_3$, we conclude that \eqref{theorem quadratic stabilizability LMI} is feasible. This proves the corollary.
\end{proof}

\subsection{Dissipativity analysis} \label{section Dissipativity analysis}
We now extend the proposed framework to data-driven dissipativity analysis. To this end, we associate with the unknown system \eqref{eq real system} an unknown performance output
\begin{equation} \label{eq output performance dissipativity}
    y(t) = C_sx(t)+D_su(t)+v(t),
\end{equation}
where $y,v \in \mathbb{R}^p$. We assume that the matrices $C_s \in \mathbb{R}^{p \times n}$ and $D_s\in \mathbb{R}^{p \times m}$ are unknown. Instead, we are given the output data, collected in the matrix
$$
Y_- := \begin{bmatrix} 
            y(0) & y(1) & \cdots & y(T-1)
	\end{bmatrix}.
$$
Together with $w(t)$ in \eqref{eq real system}, we define the unknown noise matrix
$$
	E_- := \begin{bmatrix}
            w(0) & w(1) & \cdots & w(T-1)\\
			v(0) & v(1) & \cdots & v(T-1)
	\end{bmatrix}.
$$
Extending the noise model $\mathcal{W}_{\mathrm{F}}$ to dissipativity analysis, we assume that the noise matrix $E_-$ satisfies
\begin{equation} \label{dissipativity noise bound}
    \vecrz\left(E_-\right) \in \mathcal{Z}_{(n+p)T}(\Xi),
\end{equation}
where $\Xi \in \mathbf{\Pi}_{1,(n+p)T}$ is given. Moreover, we assume that $\Xi_{22} = \xi \otimes I_{n+p}$ for some $\xi \in \mathbb{S}^T$. The system $(A,B,C,D)$ is called compatible with the data $(U_-,X,Y_-)$ if
\begin{equation} \label{dissipativity dataeq}
    \begin{bmatrix}
    X_+\\Y_-
\end{bmatrix} = \begin{bmatrix}
    A & B\\
    C & D
\end{bmatrix} \begin{bmatrix}
    X_-\\U_-
\end{bmatrix}+ E_-,
\end{equation}
holds for some $E_-$ satisfying \eqref{dissipativity noise bound}. The set of systems compatible with the data is denoted by $\Sigma_{(U_-,X,Y_-)}$, i.e.,
$$
\begin{aligned}
    \Sigma_{(U_-,X,Y_-)} := \big\{ (A,B,C,D) \;\big|\;
\text{\eqref{dissipativity dataeq} holds}& \\
 \text{for some } E_- &\text{ satisfying }\eqref{dissipativity noise bound}
\big\}\!.
\end{aligned}
$$
Clearly, $(A_s,B_s,C_s,D_s) \in \Sigma_{(U_-,X,Y_-)}$. Define
$$
\begin{aligned}
 &N := \\
    &\!\!\begin{bmatrix}
     1 & 0 \\
    \!\vecrz \! \left(\!\begin{bmatrix}
        X_+\\
        Y_-
    \end{bmatrix}\!\right)\! &\!\! -\begin{bmatrix}
    X_-\! \\ U
\end{bmatrix}^{\!\!\top}\!\! \otimes I
\end{bmatrix}^{\!\!\top}\!\Xi 
\begin{bmatrix}
     1 & 0 \\
    \!\vecrz \! \left(\!\begin{bmatrix}
        X_+\\
        Y_-
    \end{bmatrix}\!\right)\! &\!\! -\begin{bmatrix}
    X_-\! \\ U_-
\end{bmatrix}^{\!\!\top}\!\! \otimes I
\end{bmatrix}\!\!.
\end{aligned}
$$
By adapting the proof of Lemma \ref{lemma N in Pi}, $N \in \mathbf{\Pi}_{1,(n+p)(n+m)}$. A system $(A,B,C,D)$ in $\Sigma_{(U_-,X,Y_-)}$ if and only if 
\begin{equation} \label{vec ABCD in ZN}
    \vecrz\left(\begin{bmatrix}
        A & B\\
        C & D
    \end{bmatrix}\right) \in \mathcal{Z}_{(n+p)(n+m)}(N).
\end{equation}

Let $S \in \mathbb{S}^{m+p}$. A system $(A,B,C,D)$ is said to be dissipative with respect to the supply rate
\begin{equation} \label{supply rate}
    s(u,y) = \begin{bmatrix}
    u\\y
\end{bmatrix}^{\top} S \begin{bmatrix}
    u\\y
\end{bmatrix}
\end{equation}
if there exists a $P \geq 0$ such that
\begin{equation} \label{dissipativity QMI definition}
    \begin{bmatrix}
    I & 0\\
    0 & I\\
    A & B\\
    C & D
\end{bmatrix}^{\top}\begin{bmatrix}
    P & 0 & 0 & 0\\
    0 & S_{11} & 0 & S_{12}\\
    0 & 0 & -P & 0\\
    0 & S_{21} & 0 & S_{22}\\
\end{bmatrix}\begin{bmatrix}
    I & 0\\
    0 & I\\
    A & B\\
    C & D
\end{bmatrix} \geq 0.
\end{equation}
We aim to verify whether the true system $(A_s,B_s,C_s,D_s)$ is dissipative with respect to $s(u,y)$. Since we cannot distinguish the true system from any other system in $\Sigma_{(U_-,X,Y_-)}$, we need to verify dissipativity of all systems in $\Sigma_{(U_-,X,Y_-)}$. This motivates the following definition of informativity.
\begin{definition}
    The data $(U_-,X,Y_-)$ are called \emph{informative for dissipativity with respect to the supply rate \eqref{supply rate}} if there exists a $P \in \mathbb{S}^n$ with $P \geq 0$ such that \eqref{dissipativity QMI definition} holds for all $(A,B,C,D) \in \Sigma_{(U_-,X,Y_-)}$.
\end{definition}

We aim to find conditions under which the data are informative for dissipativity. To this end, we introduce the following lemma.

\begin{lemma} \label{lemma P pd}
    Assume that $N|N_{22} > 0$ and $\operatorname{In}(S) = (p,0,m)$. The data $(U_-,X,Y_-)$ are informative for dissipativity with respect to the supply rate \eqref{supply rate} if and only if there exists a $P \in \mathbb{S}^n$ with $P > 0$ such that \eqref{dissipativity QMI definition} holds for all $(A,B,C,D) \in \Sigma_{(U_-,X,Y_-)}$.
\end{lemma}

By \cite[Thm.~3.2]{DDCQMI}, $N|N_{22} > 0$ implies that the system set has a nonempty interior. Therefore, the proof of Lemma~\ref{lemma P pd} follows by adapting the proof of \cite[Lem.~4.9]{Henkbook}.

The following lemma is quoted from \cite[Lem.~4.11]{Henkbook}, which will be useful in deriving the analysis approach. 
\begin{lemma} \label{lemma S inv}
    Let $P \in \mathbb{S}^n$ and $S \in \mathbb{S}^{m,p}$, and let $(A,B,C,D)$ be any system in $\Sigma_{(U_-,X,Y_-)}$. Assume that $P > 0$ and $\operatorname{In}(S) = (p,0,m)$. Define $R := P^{-1}$ and partition $-S^{-1}$ as
    \begin{equation} \label{definition minus S inv}
        -S^{-1} := \begin{bmatrix}
        F & G\\
        G^{\top} & H
    \end{bmatrix},
    \end{equation}
    where $F \in \mathbb{S}^m$. Then, \eqref{dissipativity QMI definition} holds if and only if
    \begin{equation} \label{dissipativity QMI}
    \begin{bmatrix}
    I \!&\! 0\\
    0 \!&\! I\\
    A^{\top} \!&\! C^{\top}\\
    B^{\top} \!&\! D^{\top}
\end{bmatrix}^{\!\!\!\top}\!\!\begin{bmatrix}
    R & 0 & 0 & 0\\
    0 & H & 0 & -G^{\top}\\
    0 & 0 & -R & 0\\
    0 & -G & 0 & F\\
\end{bmatrix}\!\!\begin{bmatrix}
    I \!&\! 0\\
    0 \!&\! I\\
    A^{\top} \!&\! C^{\top}\\
    B^{\top} \!&\! D^{\top}
\end{bmatrix} \!\geq 0.
\end{equation}
\end{lemma}
\vspace{1em}

With Lemma~\ref{lemma S inv}, we apply Theorem~\ref{theorem main nonstrict} to derive the conditions for informativity. This leads to the following theorem. 

\begin{theorem} \label{theorem dissipativity}
Define 
$$
Q_{22} :=\begin{bmatrix}
   X_- \\
    U
\end{bmatrix}\xi\begin{bmatrix}
   X_- \\
    U
\end{bmatrix}^{\top},
$$
and then,
$$
    Q := \begin{bmatrix}
        (N | N_{22})I + Q_{12} Q_{22}^{\dagger} Q_{21} & Q_{12} \\
        Q_{21} & Q_{22}
    \end{bmatrix},
$$
where $\vecrz (Q_{12}) = N_{21}$. Assume that $N|N_{22} > 0$ and $\operatorname{In}(S) = (p,0,m)$. Partition $-S^{-1}$ as in \eqref{definition minus S inv}. The data $(U_-,X,Y_-)$ are informative for dissipativity with respect to the supply rate \eqref{supply rate} if and only if there exist a $R \in \mathbb{S}^n$ with $R > 0$ and a scalar $\alpha \geq 0$ such that
\begin{equation} \label{theorem dissipativity LMI}
    \begin{bmatrix}
    R & 0 & 0 & 0\\
    0 & H & 0 & -G^{\top}\\
    0 & 0 & -R & 0\\
    0 & -G & 0 & F\\
\end{bmatrix} - \alpha Q \geq 0.
\end{equation}
\end{theorem}
\vspace{1em}
\begin{proof}
    Clearly, $Q_{22}$ satisfies $N_{22} = Q_{22} \otimes I_{n+p}$. To prove the theorem, we first suppose that there exist $R$ and $\alpha$ satisfying \eqref{theorem dissipativity LMI}. Let a system $(A,B,C,D) \in \Sigma_{(U_-,X,Y_-)}$. Then, \eqref{vec ABCD in ZN} holds. It follows from Lemma~\ref{lemma general inclusion} that
    $$
    \begin{bmatrix}
        A & B\\
        C & D
    \end{bmatrix}^{\top} \in \mathcal{Z}_{n+m}(Q).
    $$
    Multiply \eqref{theorem dissipativity LMI} from left by $\begin{bmatrix}
        I & 0 & A & B\\
        0 & I & C & D
    \end{bmatrix}$ and from right by its transpose. This yields \eqref{dissipativity QMI}. Define $P:=R^{-1}$. By Lemma~\ref{lemma S inv}, \eqref{dissipativity QMI definition} holds. Since this argument holds for all $(A,B,C,D) \in \Sigma_{(U_-,X,Y_-)}$, we conclude that the data $(U_-,X,Y_-)$ are informative for dissipativity with respect to the supply rate \eqref{supply rate}.

    Next, we suppose that the data $(U_-,X,Y_-)$ are informative for dissipativity with respect to the supply rate \eqref{supply rate}. By Lemma~\ref{lemma P pd}, there exists a $P > 0$ such that \eqref{dissipativity QMI definition} holds for all $(A,B,C,D) \in \Sigma_{(U_-,X,Y_-)}$. Define $R:=P^{-1}$. By Lemma~\ref{lemma S inv}, \eqref{dissipativity QMI} holds. Using Theorem \ref{theorem main nonstrict}, there exists $\alpha \geq 0$ satisfying \eqref{theorem dissipativity LMI}. This proves the theorem.
\end{proof}

\section{Illustrative examples} \label{section Illustrative examples}
In this section, we illustrate our theoretical results with three examples. The simulations are conducted in MATLAB, using YALMIP \cite{Yalmip} with the solver MOSEK \cite{mosek}.

\subsection{Dissipativity verification}
In this example, we consider an RLC circuit as depicted in Fig.~\ref{fig:RLC}. The dynamics can be described by
$$
    \begin{bmatrix}
        \dot{V}_{C_1} \\ \dot{V}_{C_2} \\ \dot{I}_L
    \end{bmatrix} = \begin{bmatrix}
        -\frac{1}{C_1R_3} & \frac{1}{C_1R_3} & \frac{1}{C_1}\\
        \frac{1}{C_2R_3} & -\frac{1}{C_2R_3} & 0\\
        -\frac{1}{L} & 0 & -\frac{R_2}{L}\end{bmatrix}\begin{bmatrix}
        V_{C_1} \\ V_{C_2} \\ I_L
    \end{bmatrix} + \begin{bmatrix}
        0 \\ 0\\
        \frac{1}{L}
    \end{bmatrix} V.
$$
The circuit parameters are given by resistances $R_1
= R_2 = 2\mathrm{\Omega}$ and $R_3
= 1\mathrm{\Omega}$, inductance $L = 1\mathrm{H}$ and capacitances $C_1 = 1\mathrm{F}$ and $C_2 = 0.5\mathrm{F}$. We choose the input voltage $V$ as the input and the capacitor voltages $V_{C_1}$, $V_{C_2}$ together with the inductor current $I_L$ as the states. We choose the main loop current
$$
I = \frac{V}{R_1} + I_L,
$$
as the output. We apply zero-order-hold discretization with the step size $T_s = 0.1\mathrm{s}$. This leads to a discrete-time system $(A_s,B_s,C_s,D_s)$, where
$$
\begin{aligned}
A_s &= \begin{bmatrix}
    0.9092  &  0.0863 &   0.0863\\
    0.1725  &  0.8272 &   0.0085\\
   -0.0863 &  -0.0042  &  0.8145
\end{bmatrix}\!, &
B_s &= \begin{bmatrix}
    0.0045\\
    0.0003\\
    0.0905
\end{bmatrix}\!,\\
C_s &= \begin{bmatrix}
    0 & 0 & 1
\end{bmatrix}\!,\  &D_s &= 0.5.
\end{aligned}
$$

\begin{figure}[t!]
\centering
\begin{circuitikz}[scale=0.8]
\ctikzset{bipoles/length=1cm}
\draw
% Source
(0,0) to[V, l=$V$] (0,3)
      -- (2.5,3)
      to[R, l=$R_2$] (4,3)
      to[L, l=$L$] (5.5,3)
      -- (8,3)
      -- (8,2.8)
      to[R, l=$R_3$] (8,1.5)
      to[C, l=$C_2$] (8,0.2)
      -- (8,0)
      -- (0,0);

\draw
(2,3) to[R, l=$R_1$] (2,0);

\draw
(6,3) to[C, l=$C_1$] (6,0);

\end{circuitikz}
\caption{The RLC circuit.}
\label{fig:RLC}
\end{figure}
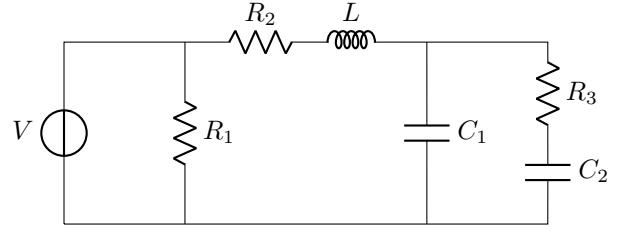

We then apply Theorem~\ref{theorem dissipativity} to verify that $(A_s,B_s,C_s,D_s)$ is passive, i.e., dissipative with respect to the supply rate $s(u,y) = uy$. In the experiment, we assume the parameter matrices are unknown. The initial state and the entries of the inputs are drawn independently and randomly from a Gaussian distribution with zero mean and variance 10. To generate noise sequences satisfying an energy constraint, we construct $\begin{bmatrix}
    w(0)^{\top} \!\!&\!\! v(0) \!\!&\!\! \cdots  \!\!&\!\! w(T-1)^{\top}  \!\!&\!\! v(T-1)
\end{bmatrix}^{\top} \in \mathbb{R}^{4T}$ by drawing a direction uniformly from the unit sphere in $\mathbb{R}^{4T}$ and a radius uniformly from $[0,c]$, where $T  \in \{ 10, 15, 20, 25, 30\}$ and $c \in \{0.1, 0.2, 0.3, 0.4\}$. For each pair $(c,T)$, we generate 100 datasets of input, state and output data samples satisfying $N|N_{22} > 0$. The conditions of Theorem~\ref{theorem dissipativity} are then verified with $\Xi$ defined by $\Xi_{11} = c^2$, $\Xi_{12} = \Xi_{21}^{\top} = 0$ and $\Xi_{22} = -I$. We record the percentage of datasets from which the corresponding inequality \eqref{theorem dissipativity LMI} is feasible. The results are collected in Table~\ref{tab:example}.

The approach in Theorem~\ref{theorem dissipativity} captures the energy constraint through a Frobenius norm bound on the noise matrix. For comparison, we apply the approach in \cite[Thm.~3]{dissipativity2022} with $\Phi_{11} = c^2 I$, $\Phi_{12} = \Phi_{21}^{\top} = 0$ and $\Phi_{22} = -I$ to the same datasets. This approach employs a QMI-based noise model that captures a spectral norm bound on the noise matrix and therefore provides an overapproximation of the energy constraint. The results are also collected in Table~\ref{tab:example}.

It is observed that as the noise level increases, the percentage of datasets for which both approaches are feasible decreases. This phenomenon arises because higher noise levels enlarge the set of systems compatible with the data, making it more difficult to find a common storage function for all data-compatible systems. Conversely, we noticed that in this example increasing $T$ leads to improved feasibility. In addition, the percentages corresponding to the Frobenius norm bound are consistently larger than those corresponding to the spectral norm bound across all 20 cases. This illustrates that the QMI-based overapproximation of energy constraints can introduce conservatism into data-driven dissipativity analysis.

\begin{table}[!t]
    \centering
    \caption{Feasibility rates (\%) for different $c$ and $T$ values}
    \label{tab:example}
    \setlength{\tabcolsep}{4.5pt} % Slightly tighter padding to fit subcolumns beautifully
    \begin{tabular}{|c|cc|cc|cc|cc|}
    \hline
     & \multicolumn{2}{c|}{$c = 0.1$} & \multicolumn{2}{c|}{$c = 0.2$} & \multicolumn{2}{c|}{$c = 0.3$} & \multicolumn{2}{c|}{$c = 0.4$} \\ 
     & Frob. & Spec. & Frob. & Spec. & Frob. & Spec. & Frob. & Spec. \\
    \hline
    $T=10$ & \textbf{48} & 44 & \textbf{18} & 12 & \textbf{4}  & 2  & \textbf{2}  & 0  \\ 
    $T=15$ & \textbf{72} & 69 & \textbf{39} & 34 & \textbf{17} & 11 & \textbf{6}  & 2  \\ 
    $T=20$ & \textbf{80} & 75 & \textbf{58} & 48 & \textbf{34} & 22 & \textbf{23} & 15 \\ 
    $T=25$ & \textbf{84} & 80 & \textbf{66} & 57 & \textbf{43} & 36 & \textbf{31} & 21 \\ 
    $T=30$ & \textbf{92} & 88 & \textbf{76} & 71 & \textbf{54} & 46 & \textbf{39} & 32 \\ 
    \hline
    \end{tabular}
\end{table}

\subsection{A numerical example on quadratic stabilization} \label{section A numerical example on quadratic stabilization}
In the following, we consider an unstable system of the form \eqref{eq real system} with $A_s$ and $B_s$ given by
$$
A_s \!=\! \begin{bmatrix}
    -0.795 &\! 0.231 &\! 0.312\\ 
     0.123 &\! 1.203 &\! -0.207\\
     -0.328 &\! 0.269 &\! 0.202
\end{bmatrix}\!,\ 
B_s \!=\! \begin{bmatrix}
    0.523 &\! -0.298\\
    0.697 &\! -1.605\\
    -0.348 &\! 0.452
\end{bmatrix}\!.
$$
We draw the initial state and the entries of the inputs independently and randomly from a Gaussian distribution with zero mean and unit variance. We collect $T = 20$ input and state data samples. During the experiment, the noise samples are drawn independently and uniformly at random from the ball $\{r \in \mathbb{R}^3 \mid \|r\|_2^2 \leq \epsilon \}$, where $\epsilon \in \{0.2, 0.3, 0.4, 0.5, 0.6\}$. For each $\epsilon$, we generate 100 datasets and record the percentage of datasets from which a stabilizing controller was found for the true system $(A_s,B_s)$ using Theorem \ref{theorem stab} with $\Phi = \Phi_{\epsilon}$. A $\mathcal{W}_{\mathrm{qmi}}(\Psi_{\epsilon})$-based approach presented in \cite[Thm.~5.1(a)]{DDCQMI} is also applied with the same datasets for comparison. The results are collected in Table~\ref{tab:example stabilization}.
\begin{table}[h]
        \centering
        \caption{Percentage of datasets yielding a stabilizing controller}
        \label{tab:example stabilization}
        \setlength{\tabcolsep}{5pt}
        \begin{tabular}{|c|c|c|c|c|c|}
        \hline
          & $\epsilon = 0.2$ & $\epsilon = 0.3$  & $\epsilon = 0.4$ & $\epsilon = 0.5$ & $\epsilon = 0.6$  \\ 
        \hline
        $\mathcal{W}_{\mathrm{F}}(\Phi_{\epsilon})$-based & $98\%$ & $82\%$ & $62\%$ & $43\%$ & $23\%$ 
        \\
        $\mathcal{W}_{\mathrm{qmi}(\Psi_{\epsilon})}$-based & $83\%$ & $38\%$ & $16\%$ & $2\%$ & $0\%$ 
        \\ 
        \hline
        \end{tabular}
\end{table}

We observe that as the noise level increases, the percentage of datasets for which both methods are feasible decrease. This phenomenon arises because higher noise levels enlarge the sets of systems compatible with the data, making it more difficult to find a single stabilizer for all systems in the sets. In addition, the $\mathcal{W}_{\mathrm{F}}(\Phi_{\epsilon})$-based method consistently achieves a higher feasibility rate than the $\mathcal{W}_{\mathrm{qmi}}(\Psi_{\epsilon})$-based method across all five cases. This observation supports the discussion in Section \ref{section Discussion on the instantaneous noise bound} regarding the reduced conservatism of the proposed approach.

\subsection{$\mathcal{H}_{\infty}$ control of a pendulum-cart system} \label{section a pendulum cart system}
We consider a linearized and discretized state-space model of an inverted pendulum with pivot mounted on a cart, as depicted in Fig.~\ref{fig:pend cart}. The original dynamic equations are given by
\begin{equation} \label{example H infty dynamic eq}
\begin{aligned}
    (M + m) \ddot{x} + b\dot{x} + ml\ddot{q}\cos(q) - ml\dot{q}^2\sin(q) = F&,\\
    l\ddot{q} - g\sin(q) = - \ddot{x} \cos(q)&,
\end{aligned}
\end{equation}
where $x,q,F \in \mathbb{R}$. We denote by $m$ and $l$ the mass and the length of the pendulum. The mass of the cart is denoted by $M$. The parameter $b$ denotes the viscous friction coefficient between the cart and the ground. The acceleration due to gravity is denoted by $g$. Let $x$ denote the horizontal displacement of the cart, $q$ the angle of the pendulum from the unstable equilibrium and $F$ the force applied to the cart. Consider $u=F$ as the input and $\zeta =\begin{bmatrix}
    \dot{x} & q & \dot{q}
\end{bmatrix}^{\top}$ as the state. In this example, we let the true parameters take 
$$
\begin{aligned}
    M &= 2 \mathrm{kg}, & m &= 1\mathrm{kg}, & b &= 0.1 \mathrm{N\cdot s^2/m}, \\
    l &=0.5\mathrm{m}, &  g &= 9.8\mathrm{m/s^2}. &  &
\end{aligned}
$$
The units of $F$, $x$, $\dot{x}$, $q$ and $\dot{q}$ are $\mathrm{N}$, $\mathrm{m}$, $\mathrm{m/s}$, $\mathrm{rad}$ and $\mathrm{rad/s}$, respectively. We linearize \eqref{example H infty dynamic eq} around $\zeta=0$ and then apply zero-order-hold discretization with the step size $T_s = 0.02\mathrm{s}$. After incorporating the noise $w(t) \in \mathbb{R}^{3}$, we obtain a linear discrete-time system model of the form \eqref{eq real system} with 
$$
A_s \!=\! \begin{bmatrix}
    0.9990 & -0.0981 & -0.0010\\
    0 & 1.0060 & 0.0200\\
    0.0020 & 0.5891 & 1.0060
\end{bmatrix} \text{ and }
B_s \!=\! \begin{bmatrix}
    0.0100\\
    -0.0002\\
    -0.0200
\end{bmatrix}\!.
$$
We denote the performance output as in \eqref{eq output performance} with
$$
C = \begin{bmatrix}
    0 & 1 & 0
\end{bmatrix} \text{ and } D = 0.
$$ 

\begin{figure}[t]
    \centering
    \begin{tikzpicture}[>=Stealth, scale=0.8]
    % Ground
\draw[thick] (-3,-0.4) -- (3,-0.4);

% Cart
\draw[thick, fill=gray!20] (-1.5,0) rectangle (1.5,3);

% Wheels
\draw[thick] (-0.6,-0.2) circle (0.2);
\draw[thick] (0.6,-0.2) circle (0.2);

% Pendulum
\coordinate (pivot) at (0,1.5);
\coordinate (mass) at (0.8,2.7);

\draw[thick] (pivot) -- (mass);
\filldraw[black] (mass) circle (0.15);

% Labels
\node[right] at (0.9,2.7) {$m$};
\node[left] at (-0.55,0.7) {$M$};

% Angle arc
\draw[->] (0,2.2) arc[start angle=90, end angle=65, radius=1];
\node at (0.27,2.5) {$q$};

% Force arrow
\draw[->, thick] (-2.5,1.5) -- (-1.5,1.5);
\node[right] at (-2.3,1.85) {$F$};

% Length dashed line
\draw[dashed] (pivot) -- (0,3);
\node[left] at (0.9,2.1) {$l$};

% x direction arrow at top right of dashed line
\draw[thick] (1.5,0.0) -- (1.5,-0.2);
\draw[->, thick] (1.5,-0.2) -- (2.5,-0.2);
\node[above] at (2.0,-0.2) {$x$};
    
    \end{tikzpicture}
    \caption{An inverted pendulum on a cart}
    \label{fig:pend cart}
\end{figure}
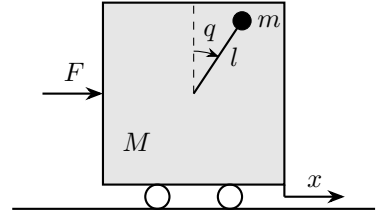

Assume that a control gain $K_0 = \begin{bmatrix}
    0.25 & 35 & 4
\end{bmatrix}$ is known such that
$A_s + B_s K_0$ is stable. We aim to design a $H_{\infty}$ controller for the true system $(A_s,B_s)$ that improves the performance of the stabilizing controller $K_0$. In the experiment, we choose the input $u(t) = K_0 \zeta(t) + u_r(t)$, where $u_r(t)$ is drawn independently and randomly from a Gaussian distribution with zero mean and unit variance. The initial state is drawn independently from the same distribution. We collect $T\in \{20,40,\dots,200 \}$ input and state data samples. The entries of the noise samples $w(t)$ are drawn independently and uniformly at random from the ball $\{r \in \mathbb{R}^3\mid \|r\|_2^2 \leq \epsilon \}$ with $\epsilon = 10^{-6}$. With the fixed $\epsilon$, we set $\Phi = \Phi_{\epsilon}$ for each $T$. We compute an $\mathcal{H}_{\infty}$ controller for the unknown true systems $(A_s,B_s)$ by applying Theorem \ref{theorem H infty}. For numerical reasons, we introduce a new decision variable $\hat{\gamma}:=-\frac{1}{\gamma^2}$ and formulate the problem as a linear optimization in $\hat{\gamma}$. The smallest performance $\gamma$ is obtained by minimizing $\hat{\gamma}$ subject to \eqref{H infty QMI} and \eqref{H infty others}. For comparison, a $\mathcal{W}_{\mathrm{qmi}}(\Psi_{\epsilon})$-based method, adapted from \cite[Thm.~8.14]{Henkbook}, is also applied using the same datasets. Since feasibility advantage of $\mathcal{W}_{\mathrm{F}}(\Phi_{\epsilon})$-based method over the $\mathcal{W}_{\mathrm{qmi}}(\Psi_{\epsilon})$-based method still holds, we restrict attention to datasets for which both methods are feasible. In particular, for each $T$, we generate 100 datasets that yields a stabilizing controller using both methods and record the average values of the resulting $\gamma$. The results are presented in Fig.~2. The optimal $\gamma$ for the closed-loop system with respect to $(A_s, B_s)$ is $5.7185$. This value is a lower bound on the best achievable $\mathcal{H}_{\infty}$ performance, which is also indicated in the figure. 

It is observed that as $T$ increases, the average $\gamma$ obtained by both methods decreases. Moreover, the $\mathcal{W}_{\mathrm{F}}(\Phi_{\epsilon})$-based method yields a lower average value of $\gamma$ than the $\mathcal{W}_{\mathrm{qmi}}(\Psi_{\epsilon})$-based method in all ten cases. 
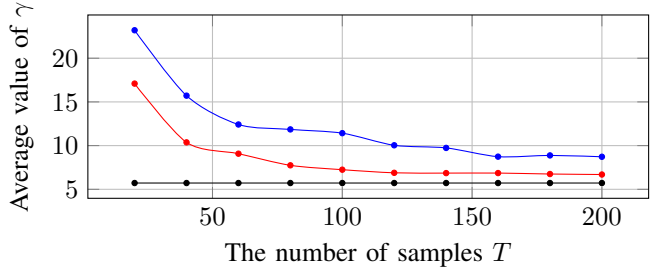
\begin{figure}[t]
    \centering
    \hspace{-0.5cm}
    \begin{tikzpicture}
\begin{axis}[
    width=9cm,        
    height=4cm,       
    xlabel={The number of samples $T$},
    ylabel={Average value of $\gamma$},
    grid=major,
    legend style={font=\normalsize},
    legend cell align=left
]

\addplot[
    smooth,
    blue,
    mark=*,
    mark size=1pt
] coordinates {
    (20,23.1998)
    (40,15.7165)
    (60,12.4163)
    (80,11.8540)
    (100,11.4231)
    (120,10.0374)
    (140,9.7363)
    (160,8.7367)
    (180,8.8786)
    (200,8.7227)
};
% \addlegendentry{$\mathcal{W}_{\mathrm{qmi}}$-based}

\addplot[
    smooth,
    red,
    mark=*,
    mark size=1pt
]  coordinates {
    (20,17.0955)
    (40,10.3711)
    (60,9.0763)
    (80,7.7449)
    (100,7.2451)
    (120,6.8875)
    (140,6.8542)
    (160,6.8586)
    (180,6.7493)
    (200,6.6921)
};
% \addlegendentry{$\mathcal{W}_{\mathrm{F}}$-based}

\addplot[
    smooth,
    black,
    mark=*,
    mark size=1pt
]  coordinates {
    (20,5.7185)
    (40,5.7185)
    (60,5.7185)
    (80,5.7185)
    (100,5.7185)
    (120,5.7185)
    (140,5.7185)
    (160,5.7185)
    (180,5.7185)
    (200,5.7185)
};
% \addlegendentry{Optimal}
\end{axis}
\end{tikzpicture}
    \caption{Average $\mathcal{H}_{\infty}$ performance as a function of the number of samples $T$. The blue curve corresponds to the $\mathcal{W}_{\mathrm{qmi}}(\Psi_{\epsilon})$-based approach, the red curve corresponds to the $\mathcal{W}_{\mathrm{F}}(\Phi_{\epsilon})$-based approach, and the black curve shows the optimal performance.}
    \label{fig:pend cart}
\end{figure}

\section{Conclusion} \label{section Conclusion}
In this work, we introduced a novel noise model and derived necessary and sufficient conditions for data-driven control with respect to this model. In particular, we considered a class of design problems that can be captured by QMIs, including quadratic stabilization and the incorporation of $\mathcal{H}_2$ and $\mathcal{H}_{\infty}$ performance specifications. We further extended the framework to data-driven analysis, offering verification conditions for quadratic stabilizability and dissipativity of all systems compatible with the data. 

In contrast to existing QMI noise models that is a generalization of spectral norm bounds, the proposed model is a generalization of Frobenius norm bounds. For instantaneously bounded noise, the new model provides a less conservative overapproximation. The reduced conservatism leads to improved feasibility in control design, while the resulting design conditions preserve the same computational benefits. 

The proposed data-driven control approaches are based on a new type of S-lemma. This S-lemma provides necessary and sufficient conditions under which a QMI is implied by a quadratic inequality in vectorized variables. To obtain computationally tractable conditions, we introduce an additional structural assumption. Relaxing this assumption to broaden the applicability of the framework remains an interesting direction for future research. 

\section*{Appendix}

\noindent\hspace{2em}{\itshape Proof of Proposition~\ref{proposition preliminaries}: }
We first focus on statement \ref{proposition statement multiple multipliers}. Suppose that there exist $\theta,\alpha_0,\alpha_1, \dots, \alpha_{T-1}$ and $\beta$ satisfying \eqref{proposition multiple multipliers LMI}. Let $(A,B)$ be a system in $\Sigma_{\mathcal{D}}(\mathcal{W}_{\mathrm{ib}}(\epsilon))$. Multiply \eqref{proposition multiple multipliers LMI} from left by $\begin{bmatrix}
    I & A & B
\end{bmatrix}$ and from right by its transpose. For $t = 0,1,\dots,T-1$, we have that $x(t+1) = Ax(t) + Bu(t) + w(t)$ and $\|w(t)\|_2^2 \leq \epsilon$. This yields that \eqref{definition 1 M theta A B} holds. Since this argument holds for all $(A,B) \in \Sigma_{\mathcal{D}}(\mathcal{W}_{\mathrm{ib}}(\epsilon))$, the data $\mathcal{D}$ are informative for $M$-control with respect to $\mathcal{W}_{\mathrm{ib}}(\epsilon)$.

Next, we prove statement \ref{proposition statement a QMI set}. The proof of the ``if" part follows by adapting the proof of \ref{proposition statement multiple multipliers}. It suffices to prove the ``only if" part. Suppose that the data $\mathcal{D}$ are informative for $M$-control with respect to $\mathcal{W}_{\mathrm{qmi}}(\Psi)$. This implies that there exists a $\theta \in \Theta$ satisfying \eqref{definition 1 M theta A B} for all $(A,B) \in \Sigma_{\mathcal{D}}(\mathcal{W}_{\mathrm{qmi}}(\Psi))$. It follows from the discussion in \cite[Sec.~5]{DDCQMI} that $(A,B) \in \Sigma_{\mathcal{D}}(\mathcal{W}_{\mathrm{qmi}}(\Psi))$ if and only if 
$$
\begin{bmatrix} I \\ A^\top \\ B^\top \end{bmatrix}^\top \begin{bmatrix}
    I \!\!&\!\! X_+ \\
    0 \!\!&\!\! -X_- \\
    0 \!\!&\!\! -U_-
\end{bmatrix} \Psi \begin{bmatrix}
    I \!\!&\!\! X_+ \\
    0 \!\!&\!\! -X_- \\
    0 \!\!&\!\! -U_-
\end{bmatrix}^{\!\!\top} 
\begin{bmatrix} I \\ A^\top \\ B^\top \end{bmatrix} \geq 0.
$$
By Lemma \ref{lemma QMI} there exist $\alpha$ and $\beta$ satisfying \eqref{proposition statement a QMI set LMI}. This proves the proposition.    
\hfill $\QED$

\section*{References}

\bibliographystyle{IEEEtran}
\bibliography{ReferencePapers}

\end{document}